\documentclass[11pt]{amsart}

\usepackage{enumerate}
\usepackage{mathtools}
\title{Square Sierpi\'nski carpets and  Latt\`es maps}%
\author{Mario Bonk}
\author{Sergei Merenkov}
\address{Department of Mathematics\\
University of California\\ Los Angeles \\CA 90095\\USA} \email{mbonk@math.ucla.edu}

\thanks{M.B.\ was supported by NSF grant DMS-1506099.}

\address{Department of Mathematics\\ City College of
New York and CUNY Graduate Center\\ New York\\ NY 10031\\USA} \email{smerenkov@ccny.cuny.edu}

\date{January 28, 2018}

\usepackage{amsmath, amssymb, amsthm,latexsym}
\usepackage{graphicx}
\newcommand\C{{\mathbb C}}

\newcommand\N{{\mathbb N}}

\newcommand\R{{\mathbb R}}
\newcommand\Ch{{\widehat {\mathbb C}}}

\newcommand\Sph{{\widehat {\mathbb C}}}

\newcommand\dist{\operatorname{dist}}
\newcommand\diam{\operatorname{diam}}

\newcommand\inte{\operatorname{int}}

\renewcommand\:{\colon}
\newcommand\sub {\subseteq}
\newcommand\ra {\rightarrow}

\def\length{\mathop{\mathrm{length}}}

\newcommand\Om{\Omega}

\newcommand\ga{\gamma}

\newcommand\eps{\epsilon}

\numberwithin{equation}{section}
\newtheorem{theorem}{Theorem}[section]
\newtheorem{proposition}[theorem]{Proposition}
\newtheorem{corollary}[theorem]{Corollary}

\newtheorem{lemma}[theorem]{Lemma}

\theoremstyle{definition}

\theoremstyle{remark}

\begin{document}


\abstract{We prove that every quasisymmetric homeomorphism of a standard square Sierpi\'nski carpet 
$S_p$, $p\ge 3$ odd, is an isometry. This strengthens and completes earlier work by the authors~\cite[Theo\-rem~1.2]{BM}. We also show that a similar  conclusion holds for quasisymmetries of the double of $S_p$ across the outer peripheral circle. Finally, as an application of the techniques developed in this paper, we prove that no standard square carpet $S_p$ is quasisymmetrically equivalent to the Julia set of a postcritically-finite rational map.}
\endabstract

\maketitle

\section{Introduction}\label{s:Intro}

The {standard square Sierpi\'nski carpet} $S_p$ is constructed as follows. We fix an odd integer  $p\ge3$. We start with the  closed unit square $Q=[0,1]^2$  in the plane $\R^2$  and subdivide it  into $p\times p$ subsquares of sidelength  $1/p$. Next, we remove the interior of the middle subsquare of this subdivision. Note that this middle subsquare is well defined since $p$ is odd. After this we repeat these two operations (i.e., subdividing and removing the middle subsquare) indefinitely on the remaining subsquares. We equip the residual set of this construction  with the Euclidean metric and call  it  the \emph{standard square Sierpi\'nski} $p$\emph{-carpet}
and denote it  by $S_p$. The sets $S_p$ are all homeomorphic to each other. In general, we call a metrizable topological space $Z$ a {\em Sierpi\'nski carpet}
if $Z$ is homeomorphic to $S_3$.

\begin{figure}
[htbp]
\includegraphics[height=35mm]{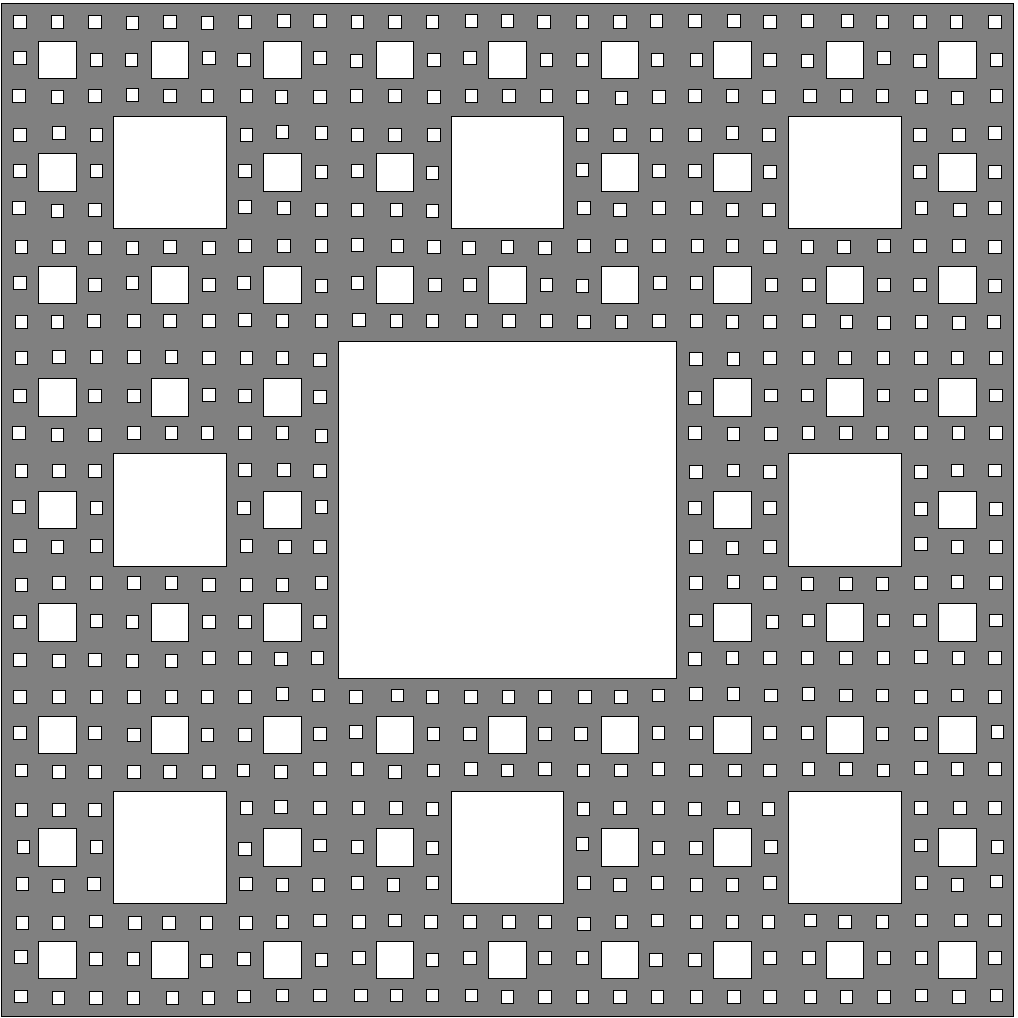}\caption{The standard square Sierpi\'nski 3-carpet $S_3$.}\label{f:ser}
\end{figure}

The boundary of $Q$ and the boundaries of all the  squares that were removed from $Q$ in the construction of $S_p$ are the so-called \emph{peripheral circles} of $S_p$. A Jordan curve $J\sub S_p$ is a peripheral circle if and only if its removal from $S_p$ does not separate $S_p$. The boundary $\partial Q$ of $Q$ is called the \emph{outer} peripheral circle of $S_p$. We denote it by $O$.

A homeomorphism $f\: X\to Y$ between metric spaces $(X,d_X)$ and $(Y, d_Y)$ is said to be \emph{quasisymmetric} or a \emph{quasisymmetry},  if there exists a homeomorphism $\eta\:[0,\infty)\to[0,\infty)$ such that
$$
\frac{d_Y(f(x),f(y))}{d_Y(f(x),f(z))}\le\eta\left(\frac{d_X(x,y)}{d_X(x,z)}\right)
$$
for all distinct points $x,y,z\in X$. If we want to emphasize a {distortion function} $\eta$, we say that $f$ is $\eta$-\emph{quasisymmetric}. 

The class of quasisymmetries  contains all bi-Lipschitz maps. The 
 composition of two quasisymmetries (when defined) and the inverse 
 of a quasisymmetry are quasisymmetric. 
So if we call two metric spaces $X$ and $Y$
{\em quasisymmetrically equivalent} if there exists a quasisymmetry
$f\:X\ra Y$, then we have  a notion of equivalence for metric spaces. 
 
 The question of when two metric spaces are quasisymmetrically equivalent has drawn much attention in recent years. This is motivated by questions in  geometric group theory, for example, such as 
 Cannon's conjecture or   the Kapovich--Kleiner conjecture which  can be reduced to quasisymmetric equivalence problems (see \cite{Bo1} for a survey of this topic). 
 
The main result of this paper  is the following statement. 

\begin{theorem}\label{thm:standard}  Every quasisymmetry  $\xi\: S_p\ra S_p$, $p\ge 3$ odd, is  an isometry. 
\end{theorem}

This  improves  results in  \cite{BM}. There it was shown that every quasisymmetry of $S_3$ is an isometry 
\cite[Theorem~1.1]{BM} and that the group of all quasisymmetries of $S_p$, $p\ge 5$ odd, is a finite dihedral group.  

The methods of~\cite{BM} do not seem to give the more general conclusion of Theorem~\ref{thm:standard} (see the discussion  in~\cite[Remark~8.3]{BM}). In the present paper we do rely on the results in 
 \cite{BM}, but for the proof of Theorem~\ref{thm:standard} we combine this with new ideas that were developed in 
 \cite{BLM} for the study of quasisymmetries of Sierpi\'nski carpets that arise as Julia sets of postcritically-finite rational maps. 
 Our methods also allow us to prove  other related  rigidity results
 for quasisymmetries.  For their formulation we require some more definitions.


 We consider the double 
$P$ of the unit square $Q$, i.e., $P$ is obtained from two identical copies  of   $Q$  glued together  by identifying corresponding points on their  boundaries. 
We refer to $P$ as a \emph{pillow} and endow it with the unique 
 path metric whose  restriction to  each of the two copies of $Q$ in $P$ coincides with the Euclidean metric.  We  can  identify one of the isometric copies of $Q$ with $Q$ itself and call it 
 the  {\em front} of $P$. Then $Q\sub P$. The other 
 isometric copy $Q'$ of $Q$ in $P$ is called the {\em back} of $P$.

We  consider $S_p$ as a subset of the front
$Q$ of $P$. The back $Q'$ of $P$  carries   another isometric copy $S'_p$ of $S_p$.
 We use the notation  $D_p=S_p\cup S'_p$ for the union of these sets and equip it  with the restriction of the path metric on $P$. Then $D_p$  is a Sierpi\'nski carpet  (this easily follows from a topological characterization of Sierpi\'nski carpets due to Whyburn \cite{Wh}). It  consists of two copies of $S_p$ glued together along the outer peripheral circle.


Our methods give  the following rigidity result for  $D_p$. 

\begin{theorem}\label{thm:double} 
Every quasisymmetry $\xi\: D_p\ra D_p$, $p\ge 3$ odd, is an isometry. 
\end{theorem}

The geometry of $S_p$ distinguishes  the peripheral circle  $O$. This is supported by the fact that for 
the investigations in \cite{BM} and also for our proof of 
Theorem~\ref{thm:standard} the starting point is 
 the non-trivial fact that every quasisymmetry $\xi\: S_p\ra S_p$ has to preserve the outer peripheral circle $O$ as a  set, i.e., $\xi(O)=O$.
 In contrast, the Sierpi\'nski carpet 
$D_p$ does not carry such a distinguished peripheral circle; this  makes the rigidity result given by Theorem~\ref{thm:double}
somewhat more surprising.

To formulate our last result, we have to briefly review some standard facts from complex dynamics (see \cite{Be} for general background). 
Let $f\:  \Sph \ra \Sph$ be a rational map on the Riemann sphere 
$\Sph$ with degree $\ge 2$. For $n\in \N$, we denote by 
$$ f^n=\underbrace{f\circ \dots \circ f}_{\text{ $n$ factors}} $$
 the $n$-th iterate of $f$.
 The \emph{Fatou set} of  $f$, denoted by $\mathcal F(f)$, is the set of all points in $\Sph$ that have neighborhoods where the sequence $\{f^n\}_{n\in \N}$ of  iterates of $f$ is a normal family. The complement of $\mathcal F(f)$  in $\Sph$ is called the \emph{Julia set} of $f$ and  denoted by $\mathcal J(f)$. It is a standard  fact that $\mathcal J(f)$ is a non-empty compact set that is completely invariant under $f$, i.e., $f^{-1}(\mathcal J(f))=\mathcal J(f)=f(\mathcal J(f))$. 

The \emph{critical set} of  $f$ consists of all points in $\Sph$ near which $f$ is not a local homeomorphism. This is a finite subset of $\Sph$. The \emph{postcritical set} 
$$ \bigcup_{n\in \N} \{ f^n(c): \text {$c$ critical point of $f$}\}$$ 
of  $f$ consists of all forward iterates of critical points. A rational map
$f$  is said to be \emph{postcritically-finite} if its  postcritical set  is finite.

In \cite{BLM} it was shown that every quasisymmetry between two Sier\-pi\'n\-ski carpets that arise as Julia sets of postcritically-finite 
rational maps is a M\"obius transformation (i.e., a fractional linear
or conjugate fractional linear map on the Riemann sphere $\Sph$).
It is a natural question whether any of the carpets $S_p$ or $D_p$ can be quasisymmetrically equivalent to such a Julia set. The following statement shows that this is never the case.

\begin{theorem}\label{thm:julia}
No Sier\-pi\'n\-ski carpet $S_p$ or  $D_p$, $p\ge 3$ odd,  is quasisymmetrically equivalent to the Julia set $\mathcal J(g)$ of a postcritically-finite rational map $g$.
\end{theorem}

Even though there is only one topological type of Sierpi\'nski carpets~\cite{Wh}, Theorem~\ref{thm:julia} shows that standard square carpets and Julia sets of post\-cri\-ti\-cal\-ly-finite rational maps are in different quasisymmetric equivalence classes. 

By the authors' earlier work~\cite{BM}  the carpets $S_p$ and $S_q$ 
for different odd integers  $p$ and $q$ are  never quasisymmetrically 
equivalent. In~\cite{Me3}, the second author proved that a Sierpi\'nski carpet that arises as  the boundary at infinity of a torsion-free hyperbolic group cannot be quasisymmetrically equivalent to a standard carpet 
$S_p$ 
 or the Julia set of a rational map. Moreover,  in~\cite{BLM} it was  shown that no Sierpi\'nski carpet Julia set of a postcritically-finite rational map is quasisymmetrically equivalent to the limit set of a Kleinian group. 
 
 To summarize, these results tell us that there are at least three quasisymmetrically distinct classes or ``universes"  of Sierpi\'nski carpets: standard square  carpets, boundaries at infinity of hyperbolic groups (or limit sets of Kleinian groups), and Julia sets of 
 postcritically-finite rational maps. Moreover, even within these universes one often encounters infinitely many quasisymmetric equivalence classes.
 
 Before we go into the details, we will discuss  some of the ideas 
 that are used in the proofs of the main results. Our main observation is that a quasisymmetry $\xi\: D_p\ra D_p$ as in Theorem~\ref{thm:double} is related to the dynamics of a Latt\`es map $T$ (depending on $p$) that is defined on the pillow $P$ and leaves the Sier\-pi\'n\-ski carpet $D_p$ forward-invariant. More precisely, we have a relation of the form  \eqref{baserel}. Once    \eqref{baserel} is established, the proofs of Theorems~\ref{thm:standard}~and~\ref{thm:double}  are completed by carefully analyzing the implications for the mapping behavior of $\xi$ in combination with known results from \cite{BM}. For the proof of Theorem~\ref{thm:julia} one derives  similar dynamical relations for 
 a quasisymmetry $\xi$ of $D_p$ or $S_p$ onto the Julia set 
 $\mathcal{J}(g)$ of a postcritically-finite rational map $g$
 (see \eqref{eqn:rat} and \eqref{eq:SpJulia}) which ultimately lead to a contradiction. 
 
 In order to establish  \eqref{baserel} we rely on a dynamical  ``blow down-blow up" procedure very similar to the one used in \cite{BLM}. This is combined with a uniformization result
 for Sier\-pi\'n\-ski carpets proved by the first author \cite{Bo} and rigidity results for Schottky maps established by the second author \cite{Me1,Me2}. 
 
 The paper is organized as follows. In Section~\ref{s:lat} we introduce 
 the Latt\`es map $T$  mentioned above and some geometric facts related to the dynamics of $T$. 
 Section~\ref{sec:adm} 
 is devoted to the resolution of some technicalities that are ultimately caused by the lack of backward invariance of $D_p$ under $T$. This relies on the concept of an {\em admissible map}
 that is introduced and studied in this section. In Section~\ref{sec:Sch}
 we review the necessary background from the theory of Schottky maps and the required rigidity results (in particular, Theorems~\ref{thm:stabil} and~\ref{thm:uniq}). In Section~\ref{sec:lattes} we prove Proposition~\ref{baseprop} that provides the crucial relation \eqref{baserel}. The proof of Theorems~\ref{thm:standard},~\ref{thm:double} and \ref{thm:julia} are then given in the two subsequent sections.

 \section{The Latt\`es map $T$}\label{s:lat} 
 Throughout this paper  $p\ge3$ is a fixed odd integer. 
Our pillow $P$ as defined in the introduction is equipped with a path metric that agrees with the Euclidean metric on the front $Q$ and on the back $Q'$ of $P$. In the following, all metric notions related to $P$ will be based on this metric. The pillow $P$  is an (abstract) polyhedral surface and so it carries a natural conformal structure making it conformally equivalent to the 
Riemann sphere. On the subsquare $[0,1/p]^2$ of the front $Q=[0,1]^2$ of $P$, we consider the map $z\in  [0,1/p]^2 \mapsto 
pz\in Q$. By Schwarz reflection this naturally extends to a  
map $T\: P\ra P$. 
Note that this  extension of $T$ to all of $P$ using Schwarz reflection is possible, because in the obvious subdivision of $P$ into $2p^2$ subsquares of equal size, each corner of every subsquare is common to an even number of subsquares in the subdivision. Of course, $T$ depends on $p$, but we suppress this from our notation. 

With the conformal structure on $P$, the map $T$ is holomorphic. 
By the uniformization theorem there is a conformal map of $P$ onto $\Ch$. Under such a conformal identification $P\cong \Ch$, the map $T$ is a rational map on $\Sph$, a so-called  {\em Latt\`es map} (see \cite[Chapter~3]{BoMy} for a detailed discussion of  Latt\`es maps from this point of view).
 Note that $T(D_p)=D_p$, i.e., $D_p$ is forward invariant under $T$, but clearly not backward invariant. 

Let $n\in \N_0$. Then each of the two faces $Q$ and $Q'$  of the pillow $P$ is in a natural way subdivided into $p^n$ squares of side length $p^{-n}$. We call a  square obtained in this way from the subdivision of $Q$ or $Q'$ a {\em tile of level $n$} or simply an {\em $n$-tile}. So there are $2p^{2n}$ tiles of level $n$.  Similarly, we call the sides of these $n$-tiles the $n$-{\em edges} and their corners the $n$-{\em vertices} (this terminology is motivated by the language in 
\cite[Section~5.3]{BoMy}).

On  each $n$-tile $X^n$ the iterate $T^n$ behaves like a similarity map and sends  $X^n$ homeomorphically to either  $Q$ or $Q'$. Here and elsewhere we use the convention that $T^0$ denotes the identity map on $P$. 
We  assign the color white or black to the  $n$-tile $X^n$ as follows: if $T^n(X^n)=Q$, then we assign to $X^n$ the color white, and if $T^n(X^n)=Q'$ the color black. Colors on $n$-tiles alternate so that two $n$-tiles sharing a side have different colors.
Therefore,  the $n$-tiles form a {\em checkerboard tiling} of $P$
(as defined in  \cite[Section~5.3]{BoMy}).
  
More generally, if $k,n\in \N_0$, and $X^{n+k}$ is an $(n+k)$-tile, then $T^n$ is a homeomorphism of $X^{n+k}$ onto the $k$-tile  $X^k\coloneqq T^n(X^{n+k})$. Moreover, $T^n$ is
 color-preserving in the sense that $X^{n+k}$ and $X^k$
have the same color.  

In general, an {\em inverse branch} $T^{-n}$ for $n\in \N$ is a right inverse of $T^n$ defined on some subset of $P$. In this paper, we will consider very specific inverse branches defined on $Q$. To define them, let $c\coloneqq (0,0)\in Q$ be the lower left corner of $Q$. Then $Z^n=[0, 1/p^n]^2$  is the  unique $n$-tile $Z^n$ with $c\in Z^n\sub Q$ and  $T^n$ sends $Z^n$ homeomorphically onto $Q$.  We define  $T^{-n}\coloneqq (T^n|Z^n)^{-1}$ and so $T^{-n}\: Q\ra Z^n$ is the unique map such that $T^n\circ T^{-n}$ is the identity on $Q$.
  
If $k,n\in \N$, then with these definitions we have $T^{-(n+k)}=
T^{-n}\circ T^{-k}$ and, if $n>k$ in addition,   
 $T^{n-k}\circ T^{-n} =T^{-k}$. This latter consistency  condition  for  inverse branches  will be important in Section~\ref{sec:lattes} (see \eqref{eq:consist}).

For some $n$-tiles  $X^n$ the  interior $\inte(X^n)$ is disjoint from  $D_p$, because $\inte(X^n)$ falls into one of the sets that were removed   from $Q$ or $Q'$ in the construction of $S_p$ and $S_p'$. We call an  $n$-tile $X^n$ {\em good} if $\inte(X^n)\cap D_p\ne \emptyset$.  
 There are precisely 
$2(p^2-1)^n$ good $n$-tiles. It follows from the self-similar construction of $S_p$ that if $X^n$ is a good white or black $n$-tile, then $D_p\cap X^n$ is a scaled copy of $S_p$. Moreover,  then 
$T^n$ is a homeomorphism of $D_p\cap X^n$ onto $S_p$ or $S'_p$, respectively.  

The inverse branches $T^{-n}$ defined above preserve the color of a tile.
Moreover, $T^{-n}$
 induces  a bijection  between the  good subtiles of $Q$ and the good subtiles of $Z^n=T^{-n}(Q)$. So in particular, if $k\in \N_0$ and $X^k\sub Q$ is a $k$-tile, then $X^{n+k}\coloneqq T^{-n}(X^k)$ is $(n+k)$-tile with the same color as $X^k$. Moreover, $X^k$ is a good tile if and only if $X^{n+k}$ is. 
 
We   will now establish  a    geometric fact about quasisymmetries and tiles 
  that will be used later (see Lemma~\ref{lem:subtile}).
  First we prove an auxiliary result.  In both
of the following  lemmas and their proofs $p\in \N$, $p\ge 3$ odd, is fixed and metric notions refer to the piecewise Euclidean metric on $P$ discussed above. 

\begin{lemma} \label{lem:simpcon}
Let $m,\ell\in \N_0$, $\ell\ge 1$, $v\in P$ be an $m$-vertex, $K$ be the union of all $m$-edges that meet $v$, and $\Om$ be the interior of the union of all $(m+\ell)$-tiles that meet $K$. Then  $\Om$ is a simply connected region that contains the open 
$p^{-(m+\ell)}$-neighborhood of $K$, but does not contain any ball  of radius $r>\sqrt 2\cdot p^{-(m+\ell)}$.  
\end{lemma}

\begin{proof} Note that unless $v$ is a corner of $P$, the set $K$ forms a ``cross" (possibly ``folded" 
if $v\in \partial Q=\partial Q'$). If $v$ is a corner of $P$, then $K$  consists of two line segments of length $p^{-m}$ meeting perpendicularly at the  common endpoint $v$. 

Obviously, $K$ is contained in $\Om$. 
Moreover, $\Om$ is connected, because   two arbitrary points $x,y\in \Om$ can be joined by a path in $\Om$ as follows. There exist $(m+\ell)$-tiles $X$ and $Y$ with $x\in X$, $y\in Y$, $X\cap K\ne \emptyset$, and $Y\cap K\ne \emptyset$. Then one runs from $x$ to a point in $x'\in   X\cap K$ along a path in $X\cap \Om$, from $x'$ along a path in $K\sub \Om$  to a point in $y'\in   Y\cap K$, and finally from $y'$ to $y$ along a path in  $Y\cap \Om$.  This shows that $\Om$ is a region. 

The region $\Om$ is simply connected, i.e., a contractible space, because  $\Om$ can be retracted to $K\sub \Om$ and $K$ is contractible. 

Let  $x\in K$ be arbitrary. Then there exists an $(m+\ell)$-edge $e\sub K$ such that $x\in e$. There are at most six $(m+\ell)$-tiles that have one of the endpoints of $e$ as a corner. The union of these tiles is a set $M$ whose interior is contained in $\Om$ and contains 
the ball $B(x, p^{-(m+\ell)})$. Hence $B(x, p^{-(m+\ell)})\sub \Om$ which implies that $\Om$ contains the open $p^{-(m+\ell)}$-neighborhood of $K$. 

Finally, every point $x\in \Om$ is contained in an  $(m+\ell)$-tile $X$ that meets $K$. Every such tile $X$ contains   a 
corner  $y \not\in \Omega$. 
For the distance of $x$ and $y$ we have $\dist(x,y)\le \sqrt2 \cdot  p^{-(m+\ell)}$. This implies that $\Omega$ cannot contain any ball of radius $r> \sqrt2 \cdot p^{-(m+\ell)}$. 
\end{proof} 

\begin{lemma} \label{lem:subtile}
Let $\xi\: P\ra P$ be a quasisym\-metry with $\xi(D_p)\sub D_p$. Then there exist numbers $r_0,N\in \N$ and $C\ge 1$  with the following properties: if $n\in \N_0$ with $n\ge N$ and $X\sub P$ is a good $n$-tile, then there exist 
a good $(n+r_0)$-tile $Y\sub X$ and a good  $m$-tile $Z$ for some $m\in \N_0$ such that 
$\xi(Y)\sub Z$ and 
\begin{equation}\label{eq:diamxiY}
\frac 1C p^{-m} \le \diam (\xi(Y)) \le C p^{-m} . 
\end{equation}      
\end{lemma}

If $A$ and $B$ are two quantities, then we write $A\asymp B$ if there exists a constant $C\ge 1$ only depending on some ambient
 parameters such that $A/C\le B\le CA$. Similarly, we write 
 $A\lesssim B$ or $B\gtrsim A$ if $A\le CB$. 

Then \eqref{eq:diamxiY} can be written as $\diam(Z) \asymp p^{-m}$, where  the  implicit  multiplicative constant $C\geq1$ is independent of $Z$. So Lemma~\ref{lem:subtile} says that $\xi(Y)$ lies in a good  $m$-tile $Z$ of  comparable size with constants of comparability independent of the initial  choice of $X$. In general, one cannot 
guarantee that the set $\xi(X)$ itself lies in a good tile of comparable size. 

\begin{proof} Let $X$ be a good  $n$-tile, where $n\in \N_0$.  Since $\xi$ is a quasisym\-metry, the image $\xi(X)$ is a ``quasi-ball". So if $x_1$ is the center of the square $X$, then $\xi(x_1)$ has a distance to the Jordan curve $J\coloneqq\xi(\partial X)$ that is comparable to 
$\diam(J)$.  Similarly, there exists a point  $x_2\in P\setminus X$ (for example, for $x_2$ we can take the center of the face of $P$ on the opposite side of $X$) such that
$\dist(\xi(x_2), J)\gtrsim \diam(J)$, i.e., we have $\dist(\xi(x_2), J)\geq \diam(J)/C$ for some constant $C\geq1$ that depends only on $\xi$. Let $y_i=\xi(x_i)$ for $i=1,2$. Then $y_1$ and $y_2$ lie in different components of $P\setminus J$. Moreover, there exists a constant $\delta>0$ independent of $n$ and $X$ such that  
 $\dist(y_i, J)>\delta \diam(J)$. This shows that  each of the two complementary components of $J$ in $P$ contains a disk of radius 
$r\coloneqq \delta \diam(J)$.

Uniform continuity of $\xi$ implies that there exists $N\in \N_0$ that depends only on $\xi$ such that if $n\ge N$, then $\diam(J)<1/3$. 
In this case,  we can choose  the largest number  $m\in \N_0$ such that 
$\diam(J)< \frac 1{3}p^{-m}$. 
Then  $\frac 1{3} p^{-(m+1)}\le \diam(J) <\frac 1{3} p^{-m}$, and so  $\diam(J) \asymp p^{-m}$. 
We can choose $\ell \in \N$ only depending on $\delta$ (and independent of $X$) such that $r=\delta \diam (J)> \sqrt 2 \cdot p^{-(m+\ell)}$.
By the choice of $\delta$, each of the two complementary components of $J$ contains a ball of radius $r>\sqrt 2\cdot p^{-(m+\ell)}$.

\smallskip
{\em Claim.} Let $E\sub P$ denote the union of all $m$-edges. Then  there exists a point $a\in J$ such that $\dist (a, E)\ge  \eps\coloneqq p^{-(m+\ell)}$.

\smallskip 
In order to prove the claim, we argue by contradiction and  assume that there is no such point. Then $J$ is contained in the open $\eps$-neighborhood of $E$. In particular, 
there exists an $m$-edge $e$ such that $\dist(e, J)<\eps$.

If $e_1$ and $e_2$ are two disjoint $m$-edges, then the connected set $J$ cannot be $\eps$-close to both of them. Indeed, if this were the case, then   it follows from 
 $\dist(e_1,e_2)\ge p^{-m}$, $\eps\le p^{-(m+1)}\le  \frac13 p^{-m}$ and 
 $\diam(J)<\frac 1{3}p^{-m}$ that  
 $$   \tfrac 1{3}p^{-m} > \diam(J)\ge \dist(e_1,e_2)-2\eps\ge \tfrac13 p^{-m}.$$  This is a contradiction. 
 
 Since $J$ cannot be $\eps$-close to two disjoint $m$-edges, 
 one of the endpoints  $v$ of $e$, which is an $m$-vertex, has the following property: if $K$ is the set of all $m$-edges that meet $v$, then $J$ is contained in the $\eps$-neighborhood of $K$. In particular, the Jordan curve $J$ is contained in the simply connected region $\Om$ as defined in Lemma~\ref{lem:simpcon} for the $m$-vertex $v$ and our choice of  $\ell$. 
 
 Then one of the two complementary components $U$ of $J$ is also contained in $\Om$, because $\Om$ is simply connected.  This is a contradiction, because  $U$ contains a ball  of radius $r=\delta \diam(J)> \sqrt 2\cdot p^{-(m+\ell)}$ by what we have seen above,  while $\Om\supseteq U$ contains no such ball  by Lemma~\ref{lem:simpcon}. The Claim follows.

\smallskip \noindent 
Since $\xi$ is a quasisymmetry, we can choose $r_0\in \N$ sufficiently large  independent of $X$  with the following property: if $Y$ is any $(n+r_0)$-tile with $Y\sub X$ and 
$Y\cap \partial X\ne \emptyset$, then 
$$ \diam (\xi(Y)) \le p^{-\ell} \diam (\xi(\partial X))= p^{-\ell} \diam (J)
<\tfrac13 p^{-(m+\ell)}. $$
Note that these tiles $Y$ are lined up along the boundary of $X$ and  cover $\partial X$. Each such tile $Y$  is a good tile, because $X$ is a good tile.

Therefore, we  can choose such a tile $Y$ so that $\xi(Y)$  contains a point $a\in J$ with $\dist(a,E)\ge p^{-(m+\ell)}$ as provided by the Claim. Then 
$$ \dist(\xi(Y), E)\ge \dist(a,E)-\diam (\xi(Y))\ge p^{-(m+\ell)}-
\tfrac13 
p^{-(m+\ell)}>0,  $$
and so  $\xi(Y)$ does not meet the union  $E$ of all $m$-edges. Since 
$\xi(Y)$ is a connected set, it must be contained in the interior of 
an $m$-tile, because these interiors are precisely the complementary components of $E$. In particular, there exists an $m$-tile $Z$ such that $\xi(Y)\sub Z$. Since $Y$ is a good tile, there exists a point $b\in \inte(Y)\cap D_p$. Then 
$$\xi(b)\in \xi(\inte(Y)) \cap\xi(D_p)\sub \inte(Z)\cap D_p.$$ 
This implies that $Z$ is a  good tile.  

Since $r_0$ is fixed and independent of $X$, the fact that $\xi$ is a quasisymmetry implies that 
 $$ \diam(\xi(Y))\asymp \diam(J) \asymp p^{-m}$$ 
 with implicit multiplicative constants independent of $X$ and $Y$. 
 It follows that we can find a suitable constant $C\ge 1$ independent of $X$ such that inequality \eqref{eq:diamxiY} is always valid. The statement follows.
\end{proof}

 \section{Admissible maps}\label{sec:adm}

In order to prove Theorems~\ref{thm:standard} and~\ref{thm:double}, we want to establish a relation between a given quasisymmetry 
$\xi\: D_p\ra D_p$ and our Latt\`es map $T$ (see Proposition~\ref{baseprop}). This relation can be obtained by arguments similar to \cite{BLM} relying  on rigidity statements for Schottky maps. 
These Schottky maps are obtained after a quasisymmetric uniformization of $D_p$ by a {\em round} Sier\-pi\'n\-ski carpet, i.e., a Sier\-pi\'n\-ski carpet in $\Sph$ all of whose peripheral circles are geometric circles. 
We will discuss the necessary results in Section~\ref{sec:Sch}.

 Unfortunately, there are some technicalities that are essentially due to the lack of backward invariance of $D_p$ under $T$ 
(see \cite[Lemma~6.1]{BLM}, where a related statement relied on backward invariance).
To work around this problem, we introduce in this section the ad hoc  notion of an {\em admissible map}. We will prove several statements about these maps  that will allow us to apply the results on Schottky maps.  We now present the details.

Let $S^2$ be a topological $2$-sphere. We think of it as equipped with an orientation and a metric $d$. Subsets of $S^2$ will carry the restriction of $d$, and so it makes sense to speak of quasisymmetries between such sets. In our applications, $S^2$ will be the pillow $P$ equipped with the piecewise Euclidean metric described earlier or the Riemann sphere $\Ch$ equipped with the chordal metric. 

Let $Z\sub S^2$ be a set and  $f\: U\ra S^2$ be a map defined on 
a 
set $U\sub S^2$.  We say that $x\in Z$ is a  {\em good point}  for $f$ and $Z$ if the following condition is true:  there exists an (open)   Jordan region $V\sub S^2$ with $x\in V$ such that $f$ is defined on $V$, the set $W=f(V)$ is also a Jordan region, and $f|V\: V\ra W$ is an orientation-preserving  quasisymmetric homeomorphism with $f(V\cap Z)=W\cap Z$.
In particular,  $f$ is then a homeomorphism of $V\cap Z$ onto $W\cap Z$.  

 Let $Z\sub S^2$ be a Sierpi\'nski carpet, and $f\: S^2\ra S^2$ be a branched covering map (for the definition of a branched covering map and more background on this topic see \cite[Chapter~2]{BoMy}). We say that $f$ is {\em admissible} for the given 
 Sier\-pi\'n\-ski carpet $Z$ if $f(Z)\sub Z$ and if there exists a set $E\sub Z$ that is contained in a union  of a finite set and finitely many peripheral circles of $Z$ such that each point $x\in Z\setminus E$ is a good point for $f$ and $Z$.
We call $E$ an  {\em exceptional set}  for $f$.    Note that $E$ is not necessarily
the complement in $Z$ of all good points, but it contains this complement. 

\begin{lemma}\label{lem:qradm}
Let $Z\sub \Ch$ be a Sierpi\'nski carpet, and $f\: \Ch\ra \Ch$ be  quasiregular map with $f^{-1}(Z)=Z$. Then $f$ is an admissible map for $Z$. 
\end{lemma} 

For the definition of a quasiregular map and some related facts in a similar context see \cite[Section~2]{BLM}. The lemma implies that if  $f\: \Ch \ra \Ch$ is a rational map and its Julia set $\mathcal{J}(f)$ is a Sierpi\'nski carpet, then $f$ is admissible for $\mathcal{J}(f)$.

\begin{proof} The statement follows from \cite[Lemma 6.1]{BLM} and its proof. 
The considerations there imply that each point in $Z$ distinct from the finitely many  critical points of $f$ is a good point for $f$ and $Z$. In particular, $f$ is an admissible map for $Z$. 
\end{proof}

\begin{lemma}\label{lem:Tadm}
The Latt\`es map $T\: P\ra P$ is admissible for $D_p$. 
\end{lemma} 

\begin{proof} We know that $T$ is a branched covering map and that $T(D_p)\sub D_p$. So we have to find an exceptional set for $T$ and the Sier\-pi\'n\-ski carpet $D_p$. 

Let $M$ be the {\em middle}  peripheral circle of $S_p$, i.e., $M$ is the boundary of the first square (of side length $1/p$) that was removed from $Q$ in the construction of $S_p$. 
Let $M'$ be the corresponding peripheral circle in the back copy $S'_p$, and $F$ be the finite set consisting of $1$-vertices, i.e., the corners of all squares that arise in the natural subdivision of $Q$ and $Q'$ into squares of side length $1/p$.  Then $F$ contains all critical points of $T$ (and actually four non-critical points of $T$, namely the  four corners of $P$). 

We claim that $E\coloneqq F\cup M\cup M'$ is an exceptional set for $T$. To see this, let $x\in D_p\setminus E$ be arbitrary. Then there exists a good $1$-tile $X$ with $x\in X$. We will assume that $X$ is white (if $X$ is black, the argument is completely analogous). We now consider two cases. 

\smallskip {\em Case 1:}  $x\in \inte(X)$. Since $X$ is white, $T|X$ is a homeomorphism from 
$X$ to $Q$. Actually, $T|X$  is a quasisymmetry, because on $X$ the map behaves like a similarity scaling distances by the factor $p$.   Then $U=\inte(X)$ and $V=\inte(Q)$ are Jordan regions 
and $T$ is quasisymmetry from $U$ onto $V$. Since $X$ is a good $1$-tile, we also have 
$T(X\cap D_p)= Q\cap D_p$ which implies that $T(U\cap D_p)=V\cap D_p$. Hence $x$ is a good point for $T$. 

 \smallskip {\em Case 2:}  $x\in \partial X$. Since $x$ does not lie in $E\supseteq F$, this point belongs to the  boundary of $X$, but is  not a corner of the square $X$. Hence there exists a unique side $e\sub \partial X$ of $X$ with $x\in e$. Moreover, since $x\not\in E \supseteq M\cup M'$, 
 the side $e$ is not contained in $M\cup M'$. Hence there exists  a unique good $1$-tile $Y\ne X$ that shares the side  $e$ with  $X$. Since $X$ is white, $Y$ is black. Let $\inte(e)$ be the set of interior points of the closed arc $e$, i.e., $e$ with its two endpoints removed. Then $x\in \inte(e)$. Moreover, 
 $$ U'\coloneqq  \inte(X) \cup  \inte(e) \cup  \inte(Y)$$ 
 is a simply connected region with $x\in U'$ that is mapped by $T$ homeomorphically onto the simply connected region 
 $$V'= \inte(Q) \cup  \inte(\widetilde e) \cup  \inte(Q').$$
 Here $\widetilde e\coloneqq T(e)$ is a common side of $Q$ and  $Q'$. We have $T(U'\cap D_p)=V'\cap D_p$, because 
 $X$ and $Y$ are good $1$-tiles. Moreover, $T|U'$ scales lengths of paths in $U'$ by the factor $p$, i.e., 
 $$\length(T\circ \ga)=p\cdot\length (\ga),$$
 whenever $\ga$ is a path in $U'$. The metric  on $P$ is a geodesic metric. 
 So these considerations imply that if  $r>0$ is sufficiently small, then 
  the open ball $U\coloneqq B(x,r)$ is a Jordan region contained in $U'$ and $T$ is a quasisymmetry of $U$ onto the Jordan region $V\coloneqq B(T(x), pr)$ such that $T(U\cap D_p)=V\cap D_p$. Hence $x$ is a good point for $T$.  
  
\smallskip Since Cases 1 and 2 exhaust all possibilities, every point $x\in D_p\setminus E$ is a good point for $T$. The statement follows.
 \end{proof}

\begin{lemma}\label{lem:pullperi}
 Let $f\: S^2\ra S^2$ be a branched covering map that is an admissible map for the Sier\-pi\'n\-ski carpet $Z\sub S^2$, and let $J\sub Z$ be a peripheral circle of $Z$. Then $f^{-1}(J) \cap Z$ is contained in a union of finitely many peripheral circles of $Z$.  
\end{lemma} 

This implies that if $E$ is an exceptional set for $f$, then $f^{-1}(E)\cap Z$ is contained in a union of a finite set and finitely many peripheral circles of $Z$.

\begin{proof} Let $A\sub Z$ be the union of all peripheral circles of $Z$. Then $A$ consist precisely of those points in $Z$ that are accessible by a (half-open) path contained in the complement of $Z$. This characterization of the points in $A$ together with the definition of a good point implies that if  $x\in Z$ is  a good point for $f$, then $x\in A$ if and only if $f(x)\in A$.  

We also need the following topological fact: if $K$ is a non-degenerate continuum (i.e., a compact connected set consisting of more than one point) and if $K$ meets a point in $Z\setminus A$ or two distinct peripheral circles of $Z$, then $K\cap (Z\setminus A)$ is an uncountable set. To see this, we collapse the closure of each  complementary component of $Z$ to a point. Then by Moore's theorem (see \cite[Theorem~13.8]{BoMy}) the quotient space obtained in this way is also a topological $2$-sphere. The image $K'$  of $K$ under the quotient map is also a compact and connected set. 
The assumptions on $K$ imply that $K'$ contains more than one point, and is hence a non-degenerate   continuum. This implies that $K'$  is an uncountable set. In particular, $K'$  will contain uncountably  many points distinct from the countably many points obtained by collapsing the  complementary components of $Z$. It follows  that $K\cap (Z\setminus A)$ 
is uncountable, as desired. 

Now let  $K$ be a connected component of $f^{-1}(J)$. Then $f(K)=J$ (this follows from  a general  fact for open and continuous  maps---see \cite[Lemma~13.13]{BoMy}; since $J$ is a Jordan curve, one can also give a simple direct argument based on path lifting).   Since $f$ is finite-to-one, it follows that there are only finitely many such components $K$ of  $f^{-1}(J)$. Each of these components $K$ is a non-degenerate continuum. 

Let $x\in Z\setminus A$ be a good point of $f$. Then $f(x)\in Z\setminus A\subseteq Z\setminus J$ by what we have seen in the beginning of the proof. In particular, $x\not \in K\subseteq f^{-1}(J)$.
Since every point in $Z\setminus A$ is a good point with finitely many exceptions, the set 
$K\cap  (Z\setminus A)$ is finite. But then actually $K\cap  (Z\setminus A)=\emptyset$, because  
otherwise  $K\cap  (Z\setminus A)$ would be  uncountable.  So $K\cap Z\sub A$. This implies  that  $K\cap Z$ is contained in a single peripheral
circle of $Z$ (or is empty), because if  $K\cap Z$ met two distinct peripheral circles, then  
$K\cap  (Z\setminus A)$ would again be an uncountable set.  

We have seen that the intersection of each of the finitely many components of $f^{-1}(J)$ with $Z$ lies in a single peripheral of $Z$. The statement follows. 
\end{proof} 

\begin{lemma}\label{lem:compadm}
Let $f,g\: S^2\ra S^2$ be two branched covering maps that are admissible maps for the Sier\-pi\'n\-ski carpet $Z\sub S^2$. Then $f\circ g$ is also admissible for $Z$.  
\end{lemma} 

\begin{proof} As a composition of two branched covering maps, $h:=f\circ g$ is also a branched covering map on $S^2$. Moreover, we have $h(Z)\sub Z$. 

Let $E$ be an  exceptional set for $f$, and $E'$ be an  exceptional set for $g$. Then by the remark after Lemma~\ref{lem:pullperi} we know that $f^{-1}(E)\cap Z$ is contained   in a union of a finite set and finitely many peripheral circles of $Z$. The same is then true for  $(E'\cup f^{-1}(E))\cap Z$. So to finish the proof, it is enough to show that each point $x\in Z\setminus (E'\cup f^{-1}(E))$ is a good point for $h$. 

By our assumptions $x\in Z\setminus E'$ is a good point for $g$, and $y\coloneqq g(x)\in Z\setminus E$ is a good point for $f$. By possibly shrinking the regions in the definition of a good point if necessary, we can find Jordan regions $U,V,W\sub S^2$ with the following properties:   $x\in U$ and  $y\in V$,  the map $g$
is a quasisymmetry from $U$ onto $V$, the map $f$ is a quasisymmetry from $V$ onto $W$, and  we have $g(U\cap Z)= V\cap Z$ and $f(V\cap Z)=W\cap Z$. Then $h=f\circ g$ is a quasisymmetry from $U$ onto $W$ and $h(U\cap Z)=W\cap Z$. This show that $x$ is a good point for $h$, as desired.  
\end{proof} 

\begin{lemma}\label{lem:Tpowers}
Let $k,n\in \N_0$ and $\xi\: P\ra P$ be a quasisymmetry with $\xi(D_p)=D_p$.
Then the map $f\coloneqq \xi^{-1}\circ T^n\circ \xi \circ T^k$
is admissible for $D_p$. 
\end{lemma} 

Note that if the homeomorphism  $\xi$ reverses orientation, then it is  not a branched covering map according to the definition given in \cite[Section~2.1]{BoMy}. Conjugation by $\xi$ still preserves the class of branched covering maps.  

\begin{proof} It is clear that $f$ is a branched covering map with $f(D_p)\sub D_p$. Moreover, it follows from Lemma~\ref{lem:Tadm} and repeated application of Lemma~\ref{lem:compadm} that the maps $T^n$ and $T^k$ are admissible for $D_p$. It is also clear that conjugation of $T^n$ by $\xi$ leads to a
branched covering map  $\xi^{-1}\circ T^n\circ \xi$  that is admissible 
for $D_p$, because $\xi$ induces a bijection on the peripheral circles of $D_p$. The statement  now follows from another  application of   Lemma~\ref{lem:compadm}. 
\end{proof}

\section{Schottky maps}\label{sec:Sch}

A \emph{relative Schottky set} $S$ in a  region $D\sub \Ch$ is a subset of $D$ whose complement in $D$ is a union of open geometric disks whose closures are contained in $D$ and are pairwise disjoint. The boundaries of these disks are called the {\em peripheral circles} of $S$. A relative Schottky set in  $D=\Ch$ is called a \emph{Schottky set}. 

Let $S$ be  a relative Schottky set and $U\sub \Ch$  be an open set. A  map $f\: U\cap S\to \Ch$ is called \emph{conformal} at a point $z_0\in U\cap S$ if the \emph{derivative} of $f$ at  $z_0$,
$$
f'(z_0)=\lim_{z\in U\cap S,\, z\to z_0}\frac{f(z)-f(z_0)}{z-z_0},
$$
exists and is non-zero. If $z_0=\infty$ or $f(z_0)=\infty$, one has to interpret this in suitable charts on $\Ch$. In order to avoid this technicality, in the following we will only consider relative Schottky sets $S$ that do not contain $\infty$ and so $S\sub \C$.

Let $S, \widetilde S\sub \C$ be two relative Schottky sets,  $U\sub \Ch$ be an open set, and $f\: U\cap S\to \widetilde S$ be a local homeomorphism. Such a map $f$ is called a \emph{Schottky map} if it is conformal at every point of $U\cap S$ and its derivative is a continuous function on $U\cap S$.

Under some mild additional assumptions quasisymmetries on relative Schottky sets are Schottky maps. More precisely, the following statement is true.
\begin{theorem} \label{thm:Smap} Let $S\sub \C$ be  a relative  Schottky set of measure zero. Suppose  $U\sub \Ch$ is open and  
 $f\: U\ra \Ch$ is  a continuous map with $f(U\cap S)\sub S$ such 
that each  point $x\in U\cap S$ is a good point for $f$ and $S$. Then 
$f|U\cap S\: U\cap S\ra S$ 
 is a Schottky map.  
\end{theorem} 

\begin{proof} A special case of this statement immediately follows from  \cite[Theorem~1.2]{Me1}. Namely, if $U\sub \C$ is a Jordan region 
with partial $\partial U\sub S$ and $f$ is an orientation-preserving quasisymmetry from $U$ onto $f(U)$ with $f(U\cap S)=f(U)\cap S$, then $f|U\cap S\: U\cap S\ra S$ is a Schottky map. 

 In the general case, it is enough to show that $f|U\cap S$ is a Schottky map locally near each point $x\in U\cap S$.  We can  reduce this to the   special case,  because $x$ is a good point for $f$ and $S$. The details of the argument are very similar to the proof of  Lemma~6.1 in \cite{BLM} and so we will only give an outline. 
 
 By our assumptions for each $x\in U\cap S$ we can find Jordan regions $V,W\sub \Ch$  with $x\in V\sub U$ such that $f|V$ is an orientation-preserving quasisymmetry of $V$ onto $W$ with 
 $f(V\cap S)=W\cap S$. We would be done if  $\partial V\sub S$.
 
 Now, if $x$ does not lie on a peripheral circle, then one can shrink $V$ suitably so that $\partial V\sub S$ (see the proof of Lemma~6.1 in \cite{BLM} for the details). 
 
 For the remaining case, suppose $x$ lies on a peripheral circle of $S$. Then $x\in \partial B\sub S$,  where $B$ is one of the complementary disks of $S$. Then one doubles the Schottky set $S$ by reflection in $C=\partial B$ to obtain a new Schottky set $\widetilde S$ that does not have $C$ as a peripheral circle. By a Schwarz reflection procedure one modifies the map $f$ 
 in $B$ to obtain a map $\widetilde f$ that agrees with $f$ in the complement of $B$ near $x$. One can then find Jordan regions 
 $V,W\sub \Ch$ such that $x\in V$, $\partial V\sub \widetilde S$, and 
 $\widetilde f$ is an orientation-preserving quasisymmetry from 
 $V$ onto $W$ with $\widetilde f(V\cap \widetilde S)= W\cap \widetilde S$. This implies that  $\widetilde f$ is a Schottky map
 $V\cap \widetilde S \ra \widetilde S$. By construction $f$ and $\widetilde f$ agree 
 on $V\cap S= (V\cap \widetilde S)\setminus B$ and map it into $S$. Hence $f|U\cap S$ is a Schottky map into $S$  near $x\in V\cap S$ as desired.  
 \end{proof}

We require the  following stabilization result. 

\begin{theorem} \label{thm:stabil}
Let   $S\sub \C$ be  a locally porous relative Schottky set,  $a\in S$,  $U\sub \Ch$ be an open neighborhood  of $a$ such that $U\cap S$ is connected,  and    $u \: U\cap S\to S$ be  a Schottky map with $u(a) = a$ that is not equal to the identity on 
$U\cap S$. 
For $n\in \N$ let $h_n \: U\cap S  \to
 S$ be a Schottky map such that for some open set $U_n\sub \Ch$ the map $h_n\: U\cap S\ra U_n\cap S$ is a homeomorphism. 

Suppose the sequence $\{h_n\}$ converges locally uniformly on $U\cap S$ 
to a homeomorphism $h\: U\cap S \to \widetilde U\cap S$, where $ \widetilde U\sub \Ch$ is an open set. Then
there exists $N \in \N$ such that $h_n = h$ on $U\cap S$ for all $n \ge N$.
\end{theorem}

This is a version of \cite[Theorem~5.2]{Me2} formulated in a way that will be convenient for our applications. Note that our assumption on $u$ implies 
$u'(a)\ne 1$ by  \cite[Theorem~4.1]{Me2}. It does not make a difference whether one allows  the open sets $U, U_n, \widetilde 
U$ to contain the point $\infty\in \Ch$  (as in our formulation) or subset of $\C$ (as required in \cite{Me2}), because $\infty\not\in S$ and so we can always delete $\infty$ from the open sets. 

We refer the reader to~\cite{Me2} for the definition of local porosity. It is easy  to check that the condition of local porosity is satisfied by $S_p$ and $D_p$ and is invariant under quasisymmetric maps. 

We also need  the following uniqueness result  \cite[Corollary~4.2] {Me2}. 

\begin{theorem}\label{thm:uniq}
Let $S\sub \C$ be a locally porous  relative Schottky set, and
$U\sub \Ch$ be an open set such that $U\cap S$ is connected.  
Suppose $f,g\: U\cap S  \ra  S$ are Schottky maps, and consider
$$
A \coloneqq  \{x \in U\cap S \: f(x) = g(x)\}.$$
If $A$ has a limit  point in $U\cap S$, then  $A =U\cap  S $ and so $f=g$.
\end{theorem}

We can apply these results in our context due to  the following fact. 

\begin{lemma} \label{betaex} There exists a quasisymmetry
$\beta\:P\ra \Ch$ such that $S\coloneqq\beta(D_p)$ is a locally porous Schottky set contained in $\C$ and $U=\beta(\inte(Q))$ is a bounded
Jordan region in $\C$ such that $U\cap S$ is connected.  Moreover, there exist a point $a \in U\cap S$ and a Schottky map  $u\: U\cap S\ra S$ such that $u(a)=a$ and  $u$ is not the identity on $U\cap  S$.
\end{lemma}

\begin{proof} We use the following uniformization theorem proved in \cite{Bo} (where the terminology is also explained): if  $Z\subseteq\Sph$ is a Sierpi\'nski carpet whose peripheral circles
are uniformly relatively separated uniform quasicircles, then there  exists a quasisymmetry $\beta\: \Sph\to \Sph$ such that $\beta(Z)\sub \Sph$ is a  {\em round}  Sierpi\'nski carpet, i.e., a Sierpi\'nski carpet whose   peripheral circles  are geometric circles
(see  \cite[Corollary~1.2]{Bo}).   Since $P$ is bi-Lipschitz equivalent to $\Ch$ and our Sierpi\'nski carpet $D_p$ has peripheral circles that are uniform quasicircles and are uniformly relatively separated,  we can apply this statement and obtain a quasisymmetry $\beta\:P\ra \Sph$ such that  $S\coloneqq\beta(D_p)\sub \Sph$ is a round Sierpi\'nski carpet.   By postcomposing $\beta$ with a M\"obius transformation if necessary, we may assume that $\beta$ is orientation-preserving, 
$ S\sub \C$ and that $U=\beta(\inte(Q))$ is a  bounded Jordan region in  $\C$.  Then  $S$ is a Schottky set. It is  locally porous, because $D_p$ is a locally porous 
subset of  $P$ and this property is preserved under quasisymmetries. 

The set  
$$U\cap S =\beta(\inte(Q))\cap \beta(D_p)=\beta(\inte(Q)\cap D_p)
=\beta(S_p\setminus O)$$ is connected as a continuous image of the connected set 
 $S_p\setminus O$.

To find a point $a$ and a map $u$ with the desired properties, we consider 
$\sigma=(1/(p+1), 1/(p+1))\in Q\sub P$. Then $\sigma\in S_p\setminus O$. Indeed, the identity  
$$
\frac1{p+1}=\sum_{k=0}^\infty\frac{p-1}{p^{2(k+1)}}
$$  shows that the $p$-ary expansion of $1/(p+1)$ has only coefficients 0 and $p-1$, and thus $\sigma$ belongs to the direct product $C_p\times C_p$, where $C_p$ is a Cantor set  constructed similarly to the standard Cantor set, but instead of subdividing $[0,1]$ into three equal parts, we subdivide it into  $p$ equal parts, remove the interior of the middle part (which is well defined because $p$ is assumed to be odd), and continue in the usual self-similar way. Now $C_p\times C_p$ is a subset of $S_p$, which implies that  $\sigma\in S_p$. Clearly, $\sigma$ does not belong to $O$, and so  $\sigma\in S_p\setminus O$. 

Actually, $\sigma$ is contained in the interior of
the $2$-tile $X\coloneqq [(p-1)/p^2,1/p]^2$. This is a good $2$-tile and $T^2$ is an orientation-preserving  quasisymmetry from $\inte(X)$ onto $\inte(Q)$ 
with $T^2(\inte (X)\cap D_p)=\inte(Q)\cap D_p=S_p\setminus O$.
 The inverse map is an  orientation-preserving  quasisymmetry 
$v\: \inte (Q)\ra \inte(X)$ with  $v(S_p\setminus O)=
\inte (X)\cap D_p$.
We have  $T^2(\sigma)=\sigma$  (essentially, this follows from 
$p^2/(p+1)\equiv1/(p+1)\mod 2$), and so $v(\sigma)=\sigma$.

We now  define $a\coloneqq \beta(\sigma) \in\beta(S_p\setminus O)=U\cap S$, and  consider the map  $\widetilde u\coloneqq \beta\circ v\circ \beta^{-1}$ defined on  $U=\beta(\inte(Q))$.  Then 
$\widetilde u$ is an  orientation-preserving    quasisymmetry of  $U$ onto the open set 
$\widetilde u(U) =\beta(\inte(X))$ with 
$$ \widetilde u(U\cap S)=  \widetilde 
u(\beta(S_p\setminus O))=\beta(\inte(X)\cap D_p)=\beta(\inte(X))\cap S=\widetilde u(U)\cap S.$$ Theorem~\ref{thm:Smap}  implies that $u\coloneqq 
\widetilde u|U\cap S$ is a  Schottky map $u\:  U\cap S\ra S$. Moreover, $u(a)=a$ and $u$ is not the identity on  $S$.
\end{proof}

\begin{corollary} \label{cor:dense} Let $\beta\: 
P\ra \Sph$ be the quasisymmetry from Lemma~\ref{betaex} with $ S=\beta(D_p)$, and  $f,g\: P\ra P$ be  admissible maps for $D_p$. Define 
$$ \widetilde f\coloneqq  \beta \circ f \circ \beta^{-1},\quad  
 \widetilde g\coloneqq  \beta \circ g \circ \beta^{-1}. $$ Then there exists a region $U\sub \Ch$ such that $U \cap S$ is a connected set that is dense in $ S$ and 
$\widetilde f, \widetilde g\:  U\cap  S\ra S$ are Schottky maps. \end{corollary}

So if we conjugate the  admissible maps $f$
and $g$ for $D_p$ by the uniformizing map $\beta$, then we obtain  Schottky maps at least on the  large part $U\cap  S $ of 
$ S$.

\begin{proof} Since  $f$ and $g$ are  admissible for $D_p$,  the maps $\widetilde f$ and $\widetilde g$ are  admissible for the Sier\-pi\'n\-ski carpet $ S=\beta(D_p)\sub \C$.  This implies that  there exist 
a finite set $F\sub  S$ and finitely many peripheral circles $J_1, \dots, J_N$ of $ S$ such that 
$E\coloneqq F\cup J_1\cup \dots \cup J_N$ is an exceptional set for $\widetilde f$ and for $\widetilde g$.   
Let $D_1, \dots, D_N$ be the closures of the complementary components of $ S$ (in $\Sph$) bounded by $J_1, \dots, J_N$, respectively. Since $ S\sub \C$, we may assume that $\infty\in D_1$. Then 
$$U\coloneqq  \Sph\setminus (F\cup   D_1\cup \dots \cup D_N)$$
is a region in $\C$. The set 
$U\cap S=S\setminus E$ is connected and  dense in $S$ (the quickest way to see this is again by an argument as in the proof of Lemma~\ref{lem:pullperi} based on Moore's theorem---we leave the details to the reader). 
Note that $\widetilde f(S), \widetilde g( S)\sub  S$ and each point in  $U\cap  S=S\setminus E$ is a good point for the maps   
$\widetilde f$ and $\widetilde g$    and the set $S$.
 Theorem~\ref{thm:Smap} implies that $\widetilde f$ and  $\widetilde g$ are Schottky maps $U\cap S\ra S$. 
\end{proof} 

\begin{corollary} \label{cor:id} Let $f,g\: P\ra P$ be admissible maps 
for $D_p$.
If there exists a set $A\sub D_p$ that is relatively open in $D_p$
such that $f=g$ on $A$, then $f=g$ on $D_p$. 
\end{corollary}

\begin{proof} If $\beta$ is the map from Lemma~\ref{betaex} and $U$ is as in Corollary~\ref{cor:dense}, 
then $A'\coloneqq U\cap \beta(A)$ is a non-empty and relatively open set in 
$U\cap S$, where $  S=\beta(D_p)$. In particular, $A'$ has a limit point 
in  $U\cap  S $, and the Schottky maps $\widetilde f, \widetilde g\: U\cap S \ra  S$ as defined in Corollary~\ref{cor:dense} agree on $A'$.   It follows from Theorem~\ref{thm:uniq} that $\widetilde f$ and $\widetilde g$ agree on $U\cap  S $ and hence on $ S$, because 
$U\cap S$ is dense in $S$. Thus $f=g$ on $\beta^{-1}(  S)=D_p$.
\end{proof}

\section{Relation to Latt\`es maps}\label{sec:lattes}

We now want to prove a crucial relation between an arbitrary  quasisymmetry $\xi\: D_p\ra D_p$ and our Latt\`es map  $T$ (recall that the odd integer $p\ge 3$ is fixed).

\begin{proposition}\label{baseprop} Let $\xi \:D_p\ra D_p$ be a quasisymmetry. Then there exist $k,n,m\in \N$ such that 
\begin{equation}\label{baserel}
T^m\circ \xi  = T^n\circ \xi \circ  T^k
\end{equation} 
on $D_p$. Here we may assume that $k,n,m$  are arbitrarily large.
\end{proposition}

The proof of this proposition will occupy the rest of this section. The main ideas for establishing the relation \eqref{baserel}  are related  to those for the proof of  the similar relation (1.2) in \cite{BLM}. 

Let $\xi\:D_p\ra D_p$ be the given quasisymmetry. Then it has a (non-unique) extension to a quasisymmetry $\xi\: P\ra P$ (this 
follows  from \cite[Proposition~5.1]{Bo} (see also \cite[Theorem 1.11]{BLM});  here it is important that $P$ is bi-Lipschitz equivalent to $\Ch$ equipped with the chordal metric and that every quasiconformal map $F\: \Ch\ra \Ch$ is a quasisymmetry). 

In order to prove \eqref{baserel}, we may assume 
that this extension $\xi$ is orientation-preserving, because otherwise 
we consider    the homeomorphism 
 $\widetilde \xi\: P\ra P$  given by 
 $\widetilde \xi=R \circ \xi$, where $R\:P\ra P$ is
  the involution on the pillow  $P$ that interchanges corresponding points on the  front and back copies of $P$. Since $R$ is an orientation-reversing  isometry on $P$ with $R(D_p)=D_p$, the map $\widetilde \xi$ is also a quasisymmetry on $P$ with 
  $\widetilde \xi(D_p)=D_p$
  and it will be orientation-preserving if $\xi$ reverses orientation. 
  Moreover,  if  we have a relation as in
 \eqref{baserel} for  $\widetilde \xi$, then a corresponding  relation for $\xi$ with the same numbers $k,n,m$ immediately follows from the identity $R\circ T=T\circ R$. So in the following we may assume that the extension $\xi\: P\ra P$ of the quasisymmetry in Proposition~\ref{baseprop} is orientation-preserving.

%

We consider the point $c\coloneqq (0,0)\in S_p\sub D_p$ (i.e., the lower left corner of $D_p$). We can find a nested sequence 
of tiles  
$X_n,\ n\in \N_0$, of strictly increasing levels $k_n\in \N_0$ with  $X_n\sub Q$ and $c\in X_n$. Then each $X_n$ is a good tile. There exists a unique branch $T^{-k_n}$ on $Q$ such that $T^{-k_n}(Q)=X_n$. These branches $T^{-k_n}$ are {\em consistent} in the sense that 
\begin{equation}\label{eq:consist} 
T^{-k_n} =T^{k_{n+1}-k_n} \circ T^{-k_{n+1}}
\end{equation} 
for $n\in \N_0$. This consistency relation is preserved if we replace the original sequence of tiles  $\{X_n\}$ (and the corresponding sequence of branches $\{T^{-k_n}\}$) by a subsequence as we will do below.  

According to Lemma~\ref{lem:subtile} we can find a good $(k_n+r_0)$-tile  $Y_n\sub X_n$ and a good tile 
$Z_n$ with $\xi(Y_n)\sub Z_n$ and $\diam(\xi(Y_n))\asymp 
\diam(Z_n)$. Here $r_0\in \N$ and the comparability constant are independent of $n$. For each $n\in \N_0$, let $Y'_n\coloneqq T^{k_n}(Y_n)$. Then $Y'_n\sub Q$ is a good tile of level $r_0$ such that $T^{-k_n}(Y'_n)=Y_n$. Since there are only finitely many  
$r_0$-tiles, we may assume, by passing to a subsequence of the original sequence $\{X_n\}$ if necessary, that the tiles $Y'_n$ are 
equal to the same good $r_0$-tile $Y$. We can find an orientation-preserving  scaling map
$\varphi\: Q\ra Y$ that maps $Q$ onto $Y$.

 Let $l_n\in \N_0$ be the level of $Z_n$. Then 
$T^{l_n}|Z_n$ is a scaling map on $Z_n$ that sends 
$Z_n$  to the front face $Q$ or the back face $Q'$ of $P$ depending on whether $Z_n$ is black or white.  
Since  
 $$\diam(\xi(Y_n))\asymp  \diam(Z_n), $$
 we then have 
  $$\diam\big (T^{l_n}(\xi(Y_n))\big)\asymp \diam(T^{l_n}(Z_n))\asymp   1, $$
 and so 
  $T^{l_n}(\xi(Y_n))$ 
has \emph{uniformly large size}, i.e.,  there exists $\alpha>0$ such that  ${\rm{diam}}(T^{l_n}(\xi (Y_n))\ge\alpha$ for all $n\in \N_0$. 

Putting this all together, for each $n\in \N_0$ we obtain 
a map 
\begin{equation} \label{eq:defhn} 
h_n\coloneqq T^{l_n} \circ \xi \circ T^{-k_n}\circ \varphi.
\end{equation} 
This is a quasisymmetric embedding of $Q$ into $P$ with uniformly large image $M_n\coloneqq T^{l_n}(\xi (Y_n))$. Since the maps 
$T^{l_n}$, $T^{-k_n}$, and $\varphi$ just scale distances, the
maps $h_n$ are {\em uniformly} quasisymmetric embeddings, i.e., there exists a distortion function $\eta$ such that 
$h_n\:Q\ra M_n$ 
is an $\eta$-quasisymmetry for each $n\in \N_0$. Note that each of the four maps on the right hand side of \eqref{eq:defhn} has the property that a point in the source space of the map lies in $D_p$ if and only if its image point lies in $D_p$. This implies that for $z\in Q$ we have 
 \begin{equation} \label{eq:inv} 
 z\in D_p  \text{ if and only if }h_n(z)\in D_p. \end{equation} 

We now invoke the  following subconvergence lemma which follows from \cite[Lemma~3.3]{Bo}.

\begin{lemma}\label{lem:Subconv}
Let $(X, d_X)$ and  $(Y, d_Y)$ be compact metric spaces, 
   and let $h_n\: X\ra Y$  be an $\eta$-quasisymmetric embedding for $n\in \N$. Suppose that there exists a constant $\alpha>0$ such that 
${\rm diam}(h_n(X))\ge \alpha$   for $n\in\N$.
  Then there exist  an increasing sequence $\{i_n\}$ in $\N$ and a quasisymmetric embedding $h\: X\to Y$ such that 
$h_{i_n} \to h$ uniformly on $X$ as $n\to \infty$. 
\end{lemma}

In our situation Lemma~\ref{lem:Subconv} implies that by passing to a subsequence if necessary, we may assume that 
$h_n\ra h$ uniformly on $Q$, where $h\: Q\ra P$ is a quasisymmetric embedding. 

We claim that the relation \eqref{eq:inv} passes to the limit, i.e., for $z\in Q$ we have 
 \begin{equation} \label{eq:inv2} 
 z\in D_p  \text{ if and only if }h(z)\in D_p. \end{equation}
Indeed, if $z\in Q\cap D_p$, then $h_n(z)\in D_p$ for each $n\in \N$ by \eqref{eq:inv}. Moreover, 
$h_n(z)\to h(z)$ as $n\to\infty$ and so $h(z)\in D_p$, because $D_p$ is a closed subset of $P$.

For the other implication, we argue by contradiction and assume 
that $z\in Q\setminus D_p$, but $h(z)\in D_p$. Let,  as before, 
 $O$ denote the boundary of $Q$. 
 Since $O\sub D_p$, the point    $z$ lies in the interior of $Q$. Then we can find a small neighborhood $W$ of $z$ with $W\sub Q\setminus D_p$. Since $h_n\to h$ uniformly on $W$, a topological degree argument implies that for sufficiently large $n$ there exists $z_n\in W$ such that $h_n(z_n)=h(z)\in D_p$. Then $z_n\in D_p$ by \eqref{eq:inv}.
This is a contradiction, because $z_n\in W\sub P\setminus D_p$. 
Relation  \eqref{eq:inv2} follows. 

Let $n\in \N$. Then for one of the 
two open Jordan regions $\Om_n\sub P$ bounded by $h_n(O)$
the map $h_n$ is a quasisymmetry of $Q\setminus O=\inte(Q)$ onto $\Om_n$. By 
\eqref{eq:inv} this implies that    
$h_n(\inte(Q)\cap D_p)=\Om_n\cap D_p$.

Similarly, there exists a Jordan region $\Om\sub P$ such that 
$h$ is a quasisymmetry of $\inte(Q)$ onto $\Om$. By \eqref{eq:inv2} we then have
$h(\inte(Q)\cap D_p)=\Om\cap D_p$.

We are now in a situation that is very similar to what was established in Step III in the proof of Theorem~1.4 in \cite{BLM}: we  want to show that the sequence 
$\{h_n\}$ stabilizes and $h_n\equiv h$ for large $n$. As in \cite{BLM}, we will invoke  rigidity statements for Schottky maps.

Let $\beta$ be  the map provided by Lemma~\ref{betaex}, and, as in this lemma, let $U=\beta(\inte(Q))$ and $S=\beta(D_p)$. Then
$S\sub \C$  is a locally porous  Schottky set. For $n\in \N$ the map 
$$\widetilde h_n \coloneqq \beta\circ h_n \circ\beta^{-1}$$
is an orientation-preserving quasisymmetry of $U$ onto $U_n\coloneqq \beta(\Omega_n)$ 
such that 
$$ \widetilde h_n (U\cap S)=\beta (h_n(\inte(Q)\cap D_p))=\beta(\Om_n\cap D_p)=U_n\cap S. $$ 
 It follows from Theorem~\ref{thm:Smap} that  $\widetilde h_n\: U\cap S\ra  S$ is a 
Schottky map, and a homeomorphism from $U\cap S$ onto $U_n\cap S$. 

By the same reasoning the map  $\widetilde h\coloneqq \beta\circ h \circ\beta^{-1}$ is a also a Schottky map $U\cap S\ra S$, and a homeomorphism of $U\cap S$ onto $\widetilde U\cap S$, where $\widetilde U=\beta(\Om)$. 
Moreover, we have $\widetilde h_n \ra \widetilde h$ locally uniformly on $U\cap S$. Theorem~\ref{thm:stabil} implies that there exists $N\in\N$ such that  
$\widetilde h_n=\widetilde h_{n+1}$ on $U\cap S=\beta(S_p\setminus O)$, and so  $h_n=h_{n+1}$  on $S_p\setminus O$ for all $n\ge N$. 

For such an $n$ we then have 
\begin{equation}\label{eq:step1}
T^{l_{n+1}}\circ \xi \circ T^{-k_{n+1}} \circ \varphi= T^{l_{n}}\circ \xi \circ 
T^{-k_{n}}\circ \varphi
 \end{equation}  
on $S_p\setminus O$, and hence on $S_p$ by continuity.  
Since $\varphi$ is a homeomorphism of $S_p$ onto $Y\cap D_p$, this leads to 
$$T^{l_{n+1}}\circ \xi \circ T^{-k_{n+1}} = T^{l_{n}}\circ \xi \circ 
T^{-k_{n}}$$
on $Y\cap D_p$. By the consistency relation \eqref{eq:consist} this gives 
$$
T^{l_{n+1}}\circ \xi \circ T^{-k_{n+1}} = T^{l_{n}}\circ \xi \circ 
T^{k_{n+1}-k_n} \circ  T^{-k_{n+1}}
$$ on $Y\cap D_p$, and so 
\begin{equation}\label{eq:fund}
\xi^{-1}\circ T^{l_{n+1}}\circ \xi = \xi^{-1}\circ  T^{l_{n}}\circ \xi \circ 
T^{k_{n+1}-k_n}
 \end{equation}  
on $T^{-k_{n+1}}(Y\cap D_p)=Y_{n+1}\cap D_p$. 

Since $Y_{n+1}$ is a good tile, the set $Y_{n+1}\cap D_p$ is a subset of $D_p$ with non-empty relative interior.   We can now  apply Lemma~\ref{lem:Tpowers} and Corollary~\ref{cor:id} to conclude that \eqref{eq:fund} is true on the whole set $D_p$. Here we can cancel $\xi^{-1}$. Since $k_{n+1}>k_n$ by our initial choice of the tiles $X_n$, it follows that there exist numbers $k,m,n\in \N$  such that
\begin{equation} \label{eq:fund2}
T^{m}\circ \xi = T^{n}\circ \xi \circ T^{k}
\end{equation}
on $D_p$. We can make $m$ and $n$ arbitrarily large by  postcomposing both sides in this identity with iterates of $T$. Moreover, using this equation for a large enough multiple of the original $n$ in \eqref{eq:fund2}  and precomposing with iterates of $T^k$, we can also make $k$ in \eqref{eq:fund2}  arbitrarily large. So for each $N'\in \N$ we can find $k,n,m\ge N'$ so that 
\eqref{eq:fund2} holds. 
This shows that  Proposition~\ref{baseprop} is indeed true.

\section{Proof of Theorem~\ref{thm:standard}}\label{sec:11}

We first state two relevant results from~\cite{BM}. 
The following statement is  part of~\cite[Lemma~8.1]{BM}.
\begin{theorem}\label{thm:distPair}
 Every quasisymmetry $\xi \: S_p\ra S_p$,  $p\ge3$  odd,  preserves the outer peripheral circle  $O$ setwise and so $\xi(O)=O$.
\end{theorem}

We  need the following special case of~\cite[Theorem~1.4]{BM}.

\begin{theorem} \label{thm:corners}
Let  $\xi\: S_p\ra S_p$,  $p\ge 3$ odd,  be an orientation-preserving quasisymmetry  that fixes the four corners of $Q$. Then $\xi$ is the identity on $S_p$.
\end{theorem}

Here  $\xi \: S_p\ra S_p$ is {\em orientation-preserving}
or {\em orien\-tation-reserving} if it has an extension to a homeomorphism on $\Sph\supseteq S_p$ with the same property. Every 
quasisymmetry $\xi\: S_p\ra S_p$ is either orientation-preserving or orientation-reversing.

We can now prove our first main result.

\begin{proof}[Proof of Theorem~\ref{thm:standard}] Let  
 $\xi\: S_p\ra S_p$ be a quasisymmetry. In order to show that $\xi$ is an isometry,  we can freely pre- or postcompose $\xi$ with  isometries of $S_p$ without affecting our desired conclusion. 
 
  Without loss of generality we may assume  that $\xi$ is orientation-preserving; otherwise, we consider the  composition of  $\xi$ with a reflection of $S_p$  in, say, one of the diagonals of the square  $Q$.

By Theorem~\ref{thm:distPair} we know that $\xi(O)=O$ and so $\xi$ restricts to  an orien\-ta\-tion-preserving homeomorphism on $O$. 
If $\xi$ send each corner of $Q$ to another corner of $Q$, then
 $\xi$ has to preserve the cyclic order of these corners. This implies that on the set of corners $\xi$ acts as a rotation  by  an integer multiple of $\pi/2$ around the center of $Q$.
It follows that we can postcompose $\xi$ with such a  rotation of $S_p$ so that the new quasisymmetry  actually fixes the corners of $Q$. By Theorem~\ref{thm:corners} this map is then the identity on $S_p$ and we conclude that $\xi$ is an isometry as desired.

So we are reduced to the case where $\xi$ sends a corner of  $Q$ to a non-corner point. Equivalently, the preimage of a corner  of $Q$ under $\xi$ is not a corner point.  Again, by using rotations of $S_p$, we may assume that the preimage $q=\xi^{-1}(c)\in O$ of the lower left corner  $c=(0,0)\in \R^2$ of $S_p$  is not a corner.

If $q$ does not lie in the open bottom side $(0,1)\times\{0\}$ of $Q$, we can precompose $\xi$ with two reflections $R_1, R_2$ in appropriate symmetry lines of $Q$ (i.e., diagonals and  lines through the centers of two opposite sides of $Q$) so that $\xi$ is still orientation-preserving and $$q=(R_1\circ R_2\circ \xi^{-1})(c)\in (0,1)\times\{0\}\sub S_p. $$

Based on these considerations, we reduced to the case
that $q=\xi^{-1}(c)\in (0,1)\times\{0\}\sub S_p$.  We want to derive a contradiction from this statement, which will establish the theorem.

Let $R$ be the involution  on $P$ that  
interchanges corresponding points in the two copies of $Q$.
Since $\xi$ preserves the outer peripheral circle $O$ of $S_p$, it induces a homeomorphism of $D_p$ that agrees with $\xi$ on the front copy of $S_p$ and is given by $R\circ \xi\circ R$ on the back copy of $S_p$ in $D_p$. We continue to denote this map on $D_p$ by $\xi$. 
It is clear that $\xi\: D_p\ra D_p$ is a homeomorphism. 

If $D_p$  is, as before, endowed with the restriction of the path metric on the pillow $P$, then the  map  $\xi\: D_p\ra D_p$  
is actually a  quasisymmetry. 
To see this, first note that  the original map $\xi$ on $S_p$ extends to a quasiconformal homeomorphism of the unit square $Q$ (see, e.g., \cite[Proposition~5.1]{Bo}). By using this and the reflection $R$,  we can find a  quasiconformal homeomorphism  on  $P$ that extends the homeomorphism $\xi\: D_p\ra D_p$. We call this new map on $P$ also $\xi$. Now $P$ is bi-Lipschitz equivalent to $\Sph$. If we conjugate $\xi\: P\ra P$ by such a bi-Lipschitz map, then we obtain a quasiconformal homeomorphism $\widetilde \xi$ on $\Sph$.  Each quasiconformal map on $\Sph$ is a quasisymmetry. It follows that $\widetilde \xi$ is a quasisymmetry. Hence its  bi-Lipschitz conjugate $\xi\: P\ra P$ is also a quasisymmetry, and so is the restriction $\xi\: D_p\ra D_p$. 

Let $T\:P\ra P$ be the Latt\`es map defined in Section~\ref{s:lat}.  Then it follows from  Proposition~\ref{baseprop} that we have a  relation of the form  
\begin{equation}\label{eq:base3}
 T^m\circ \xi=T^n\circ \xi\circ T^k
 \end{equation}
on $D_p$, where we may assume that  $k,n,m\in \N$ are suitably large.

 Let $I\sub (0,1)$ be a small open interval one of whose endpoints is $q$, such that $\xi(I)$ is completely contained in one of the sides of $Q$. Recall that $\xi(q)=c=(0,0)$. Each side of $Q$ is forward invariant under $T$. This is clearly true for the two sides of $Q$ 
 that contain the origin $(0,0)$. It is also true for the two remaining sides since $p$ is odd.
For a large enough $k$  we then have $T^k(I)=[0,1]$, because, roughly speaking, $T$ expands by the factor $p$. We may assume that \eqref{eq:base3} holds for this $k$. Since
$(0,0)=\xi(q)\in \xi((0,1))$, the set $A\coloneqq \xi([0,1])$ meets the interior of at least two sides of $Q$.  
It follows that
$$(T^{n}\circ \xi \circ T^{k})(I)=(T^{n}\circ \xi)([0,1])=T^n(A)$$ meets the interior of at least two sides of $Q$. 

On the other hand, since  $\xi(I)$ is contained in one side of $Q$ and $T$ leaves each side of $Q$ invariant,  we conclude  
from  \eqref{eq:base3} that $T^n(A)= (T^{m}\circ \xi)(I)$
is contained in one side of $Q$. Therefore, $T^n(A)$ cannot meet the interior of two different sides of $Q$. This is a contradiction. 
\qed

\section{Proofs of Theorem~\ref{thm:double} and Theorem~\ref{thm:julia}}\label{sec:1213}

For the proof of Theorem~\ref{thm:double} we require 
auxiliary results about  certain weak tangent spaces at points of $D_p$. We will discuss these statements first.
  For a more detailed treatment of similar weak tangents for $S_p$ the reader may consult~\cite[Section~7]{BM}.

Let $(X, d)$ be a  metric space. Then  
a \emph{weak tangent} of $X$ at a point $a\in X$ is the  Gromov--Hausdorff limit of a sequence of pointed  metric spaces $(X,a, \lambda_n d)$, where $\lambda_n>0$ and $\lambda_n\to \infty$ as $n\to \infty$ (for the relevant definitions and general background see \cite[Chapters 7~and 8]{BBI}). We are interested in weak tangents of subsets of $D_p$ equipped with the restriction of the piecewise Euclidean metric on the pillow $P$. In this case, it is convenient to restrict the scaling factors $\lambda_n$ in the definition of  weak tangents
to   powers of $p$, i.e., they have the form $\lambda_n=p^{k_n}$ with $k_n\in \N_0$ and $k_n\to \infty$ as $n\to \infty$.

For example, let   $c=(0,0)\in S_p$  be the lower left corner of $Q$.   Then  $S_p$ has an essentially  unique weak tangent at $c$ isometric to  
the union 
\begin{equation} \label{weakt}
W\coloneqq \bigcup_{n\in\N_0}p^n S_p\sub  \C
\end{equation}
with base point $0\in \C$. 
In this union we consider $S_p$ as a subset of $\R^2\cong \C$ and use the notation 
$\lambda A=\{ \lambda z: z\in A\}$ whenever $\lambda \in \C$ and $A\sub \C$.   
In particular, $W$ is a subset of the  first quadrant of $\R^2\cong \C$.

Let $q=\tfrac1{2p} (p-1, p-1)\in S_p$. Then at $q$ the set $S_p$ has  an essentially unique weak  tangent, denoted by  $\widetilde W$. It is obtained as the union of three copies of  $W$. Up to isometry, we have \begin{equation} \label{weakt2} \widetilde W\coloneqq W \cup i W \cup (-1)   W, 
\end{equation} 
with base point $0\in \C$.
Of course, $W$ and $\widetilde W$ depend on $p$, but we suppress this from our notation.

Standard compactness arguments imply that a quasisymmetric map of $D_p$ that takes a point $a\in D_p$ to another point $b\in D_p$ induces a quasisymmetric map between appropriate weak tangents at $a$ and $b$, respectively. This observation along with the following lemma will help us to eliminate  certain mapping possibilities.
\begin{lemma} \label{lem:noqs}
There is no quasisymmetry from $W$ onto $\widetilde W$ that fixes $0$. 
\end{lemma}
This follows from \cite[Proposition~7.3]{BM}.  
Note that in \cite{BM}
the setup was slightly different, because weak tangents were considered as Hausdorff limits of sets in $\Ch$ under blow-ups 
by  scaling maps. This is equivalent to our definition with the only difference that our weak tangents do not contain the point $\infty\in \Ch$ as in \cite{BM}.

\medskip

\noindent
\emph{Proof of Theorem~\ref{thm:double}.} This reduces to Theorem~\ref{thm:standard} if we can show that $\xi(O)=O$, where as before $O$ denotes common boundary of the two copies of the unit square that form the pillow $P$. In the following, we will rely on some mapping properties of $\xi$ and the 
Latt\`es map $T$ (as defined in Section~\ref{s:lat}) that we state  first. 

Since $\xi$ is a quasisymmetry on $D_p$ and hence a homeomorphism, it  maps peripheral circles
of $D_p$ to peripheral circles. Note that $O$ does not meet any peripheral circle of $D_p$. If $z\in D_p$ lies on a peripheral circle, then each preimage of $z$ under any iterate  of  $T$ also lies on a peripheral circle. 
If $C$ is a peripheral  circle  and $l\in \N$, then either $T^l(C)$ is also a peripheral circle or 
$T^l(C)=O$. 

The \emph{middle peripheral circle} $M$ of $S_p$ is the boundary of the subsquare that is removed in the first stage of the construction of $S_p$; this is the only peripheral circle of $S_p$ other than the outer peripheral circle $O$ that is invariant under the isometries of the square $Q$. Note that  $T(M)=O$.

By Proposition~\ref{baseprop} we have an equation of the form 
\begin{equation} \label{eq:fund3}
T^{m}\circ \xi = T^{n}\circ  \xi\circ T^{k} 
\end{equation}
  on $D_p$, where $k,n,m\in \N$. Here we may assume that  $k,n,m$ are  as large as we wish. 
   Then $$T^m(\xi(M))=T^n(\xi(O)), $$ because $T^k(M)=O$ for any $k\in \N$. 
  Note that $\xi(O)$, and hence $T^n(\xi(O))$, does not meet any  peripheral circle of $D_p$.
  Therefore, the set $T^m(\xi(M))$ does not meet any peripheral circle either. But $\xi(M)$ is a peripheral circle, and thus
  $T^m(\xi(M))=O$. Hence $T^n(\xi(O))=O$, and so $\xi(O)\sub T^{-n}(O)$. Note that $T^{-n}(O)$ forms a ``grid" in $P$ consisting  of 
  all $n$-edges.   
  
If $p=3$, then it is not hard to see that the only Jordan curve (such as $\xi(O))$ that is contained in 
$T^{-n}(O)$  and does not meet any peripheral circle is equal to $O$.  In this case, we conclude that $\xi(O)=O$, as desired.  To give an argument that is valid for any  odd $p\ge 3$,  more work is required. 

Let $C\coloneqq \xi(O)\subseteq T^{-n}(O)$. Then $C$ can be considered as a  polygonal loop 
 consisting of $n$-edges.  If we run through $C$ according to   some orientation,  then two successive $n$-edges  on  $C$ 
have an endpoint $a$ in common, where they meet at a right angle or at the angle 
$\pi$. If they meet at a right angle, then we call $a$ a \emph{turn} of $C$.  

\smallskip {\em Claim.}  Every turn  $a$ of $C$ must have one of the four  corners of $O$ as a preimage under $\xi$. 

\smallskip To see this, 
  we argue by contradiction and  assume that there exists  $b\in O$ that is not a corner of $O$  such that $a=\xi(b)$ is a turn of $C$.  Let $I$ be a small open interval contained in $O$, one of whose endpoints is $b$. Here we may assume that $I$ is completely contained in one side of $O$. We now apply both sides of~\eqref{eq:fund3} to $I$. We choose $k$ in this equation so large that $T^k(I)$ is the whole side of $O$ that contains $b$. Then $\xi(T^k(I))$ contains a neighborhood of $a$ in $C$. If  $n$ is large enough, as we may assume, then  $T^n(\xi(T^k(I)))$ contains at least two sides of $O$. Here it is important that $a$ is a turn of $C$.  On the other hand, if $I$ is small enough, then $T^m(\xi(I))$ is contained in one side of $O$. This is a contradiction and the Claim follows. 
  
 \smallskip 
There are now two cases to consider, depending on whether $C$ does or does not have turns. 

\smallskip 
{\em Case 1:}  $C$ has no turns.  Then $C=\xi(O)\ne O$ and $C$  runs ``parallel" to one of the sides of $O$ in $Q$ and in the back copy $Q'$ of $Q$. In particular,  
the involution $R$ (that interchanges corresponding points of $Q$ and $Q'$) is a quasisymmetry on $D_p$ preserving $C$.  Consider
$$
g= \xi^{-1} \circ R\circ \xi.
$$
Then $g$ is a quasisymmetry on $D_p$ with $g(O)=O$. Its fixed point set is the Jordan curve $\xi^{-1}(O)\neq \xi^{-1}(C)=O$.

It follows from Theorem~\ref{thm:standard} that  every  quasisymmetry on $D_p$ that preserves $O$ as a set is an isometry of the double $D_p$. In particular, $g$ must be such an isometry. There are 16 such maps: eight isometries that preserve the front and back  copies of $S_p$ and eight obtained by composing these maps by the involution $R$ that interchanges the front  and  the back copies.
Among these 16 maps there is exactly one, namely the involution $R$, whose fixed point set is a Jordan curve $J$.
In this case $J=O$. In all other cases,  the fixed point set is either empty,  finite,  a Cantor set, or all of $D_p$. On the other hand, $g$ has the fixed point set $\xi^{-1}(O)\neq \xi^{-1}(C)=O$.
 This is a contradiction, showing that Case 1 is impossible. 
 
\smallskip 
{\em Case 2:} $C$ has at least one turn.  We claim that such a turn must be a corner of $O$. To reach a contradiction,  suppose  $C$ has a turn $a$ other than a corner of $O$. Since $C=\xi(O)\subseteq T^{-n}(O)$ does not meet any peripheral circle of $D_p$, the point $a$ is the common corner of four good $n$-tiles. 
The curve $C$ is the common boundary of two complementary regions of the pillow $P$. Since $a$ is a turn of $C$, one of these regions, which we denote by $U$, has angle $3\pi/2$ at $a$ (the other region has the angle $\pi/2$ at $a$). 

There are three good $n$-tiles, i.e., three  copies of $S_p$ scaled by the factor $1/p^n$  that are contained $\overline U$ and that share $a$ as a corner. In fact, there are infinitely many such triples of copies of $S_p$ that meet at $a$ rescaled by the factor $p^{-k}$ with $k\ge n$. This implies $\overline U\cap D_p$ has a unique weak tangent at  $a$ and that it is isometric to the set  $\widetilde W$ in \eqref{weakt2} with basepoint $0$.  

Since $a$ is a turn of $C$, there exists a corner $b$ of $O$ with 
$\xi(b)=a$. Now $\xi(O)=C$,  and so either $S_p$ or the back copy $S'_p$ of $S_p$ is mapped to 
$\overline U\cap D_p$ by the quasisymmetry $\xi$. In both cases we get an induced basepoint-preserving quasisymmetry of the weak tangents of the source set  at $b$ and the image set  
$\overline U\cap D_p$ at $a$. Both $S_p$ and $S'_p$ have unique weak tangents at any corner $b$ of $O$ isometric to $W$ in \eqref{weakt} with basepoint $0$.  This implies that there exists a quasisymmetry 
 between the pointed metric spaces  $(W,0)$  and  $(\widetilde W,0)$.  This is impossible by Lemma~\ref{lem:noqs}, and so we reach a contradiction. 
 
  We conclude that  $C=\xi(O)$ has at least one turn, and that every turn of $C$ is a corner of $O$; but then necessarily $C=\xi(O)=O$, and we are done. 
\qed

\medskip

\noindent
\emph{Proof of Theorem~\ref{thm:julia}.} We argue by contradiction and assume that there exists a quasisymmetry $\xi \: D_p\ra  \mathcal J(g)$, where $p\ge 3$ is odd, and $\mathcal J(g)$ is the Julia set of a postcritically-finite rational map $g$ on $\Sph$. We will later consider the case when $S_p$ is quasisymmetrically equivalent to $\mathcal J(g)$. 

The quasisymmetry $\xi$ can be 
extended (non-uniquely) to a quasisymmetry, also called $\xi$, of the pillow $P$ onto $\Sph$. We may also assume that $\xi$ is orientation-preserving, because otherwise we can precompose 
$\xi$ with an orientation-reversing isometry of $P$ that leaves 
$D_p$ invariant (such as the reflection $R$ that interchanges 
corresponding points on the front and back face of $P$). 

As before,  we denote by $T$ the Latt\`es map as introduced in Section~\ref{s:lat} (for given $p$). The main idea of the proof now is to establish an analog of Proposition~\ref{baseprop} for the maps $g$, $\xi$, and $T$. Namely, we want to show that there exist $k,n,m\in \N$ such that 
\begin{equation}\label{eqn:rat}
g^m\circ \xi = g^n\circ \xi \circ T^k
\end{equation}
 on the set $D_p$. 

Once \eqref{eqn:rat} is established, we obtain a contradiction as follows. Since $\mathcal J(g)$ is homeomorphic to $D_p$, the set $\mathcal J(g)$ is a Sierpi\'nski carpet. Let $A$ be the set of all points in $\mathcal J(g)$ that lie on a peripheral circle of  $\mathcal J(g)$. If $J$ is a peripheral circle of  $\mathcal J(g)$, then $g(J)$ is also a peripheral circle and $g^{-1}(J)$ consists of finitely many peripheral circles (see \cite[Lemma~5.1]{BLM}). This implies that  $A$ is completely invariant under $g$, and hence under all iterates of $g$, i.e., $g^l(A)=A=g^{-l}(A)$ for each $l\in \N$. Note also that the homeomorphism $\xi$
sends the peripheral circles of $D_p$ to the peripheral circles of $\mathcal J(g)$.

 We now apply both sides of \eqref{eqn:rat} to the middle peripheral circle $M$ of $S_p\sub D_p$. 
Then the left hand side shows that $(g^m\circ \xi)(M)\sub A$. On the other hand, if we consider the right hand side of \eqref{eqn:rat}, we first note that $T^k(M)=O$. It follows that  
   $$(g^n\circ \xi \circ T^k)(M)=(g^n\circ \xi)(O)$$ is disjoint from $A$, because $O$ does not meet any peripheral circle of $D_p$ and so $\xi(O)$ and $(g^n\circ \xi)(O)$ are  disjoint from $A$.
    This is a contradiction.

In order to establish  \eqref{eqn:rat}, one uses   ideas as in the proof of Proposition~\ref{baseprop} combined with   arguments for the proof of the similar relation (8.4) in 
\cite[Section~8]{BLM}, where $T$ plays the role of the rational map $f$ and $D_p$ the role of the Julia set $\mathcal J(f)$.

As in the proof in \cite[Section~8]{BLM}, we want to implement a ``blow down-blow up" argument 
applying  ``conformal elevators". Namely, one first  uses  inverse branches $T^{-n}$ to blow down, and then iterates of $g_\xi:= \xi^{-1} \circ g \circ \xi\: P\ra P$ to blow up. Note that 
$D_p$ is completely invariant under the map $g_\xi$ and its iterates.

As in the proof of Proposition~\ref{baseprop}, we choose  a sequence $T^{-k_n}$ of inverse branches mapping the front side $Q$ of the pillow to a good tile 
$X_n\sub Q$ containing the corner $c=(0,0)\in Q$. These  branches  are consistent as in \eqref{eq:consist}. 
Each map  $T^{-k_n}$ is a scaling map, and in particular a quasisymmetry of $Q$ onto $X_n\sub Q$.  We also have   ${\rm diam}(T^{-n}(Q))\to 0$ as $n\to\infty$. All of this  is similar to the choice of inverse branches in~\cite[Section~8, Step II]{BLM}, but easier.

Following the argument of ~\cite[Section~8, Step II]{BLM}, we use the expansion property of $g$ to find for each natural number $n\in \N$, a corresponding number $l_n\in\N$ such that the sets 
 $(g_\xi^{l_n}\circ T^{-k_n})(Q)$ have uniformly large  size independent of $n$.  
Now we argue as in \cite[Section~8, Step III]{BLM} and consider the maps $\widetilde h_n\coloneqq  g_\xi^{l_n}\circ T^{-k_n}\: Q\ra P$ for $n\in \N$. Under a conformal identification $P\cong \Sph$ these are actually uniformly quasiregular maps on the region $U\coloneqq Q\setminus O$. By passing to a subsequence if necessary, we may assume that there exists a (non-constant) quasiregular map $\widetilde h\: U\ra P$ such that $\widetilde h_n\to \widetilde h$ locally uniformly on $U$ as $n\to \infty$. 

The map $\widetilde h$ is locally quasiconformal on $U$ away from the   branch points of $\widetilde h$. These branch points form a set with no limit points in $U$. This implies that we can find a point $q\in U\cap D_p$ and a small radius $r>0$ such that $B(q,2r)\sub U$ and $\widetilde h$ is quasiconformal on $B(q,2r)$. It follows that on the smaller ball 
$B(q,r)$ the maps $\widetilde h_n$ are  quasiconformal for large $n$. By discarding finitely many of the maps $\widetilde h_n$ if necessary, we may assume that they are quasiconformal for all $n\in \N$. 

  Since $q\in D_p$, we can find a good tile $Y\sub B(q,r)$ (as defined in Section~\ref{s:lat}).
Since a quasiconformal map is a local quasisymmetry,
 we conclude that  the maps $\widetilde h$ and $\widetilde h_n$ for $n\in \N$ are quasisymmetric embeddings of $Y$ into $P$. We are now in a similar situation as in the proof of Proposition~\ref{baseprop}. We choose 
an orientation-preserving scaling map $\varphi$ that sends  $Q$ onto $Y$. Note that then $\varphi(Q\cap D_p)=Y\cap D_p$.  We now define 
$$ h\coloneqq \widetilde h\circ \varphi,  $$
$$ h_n\coloneqq \widetilde h_n\circ \varphi=  
g_\xi^{l_n}\circ T^{-k_n}\circ \varphi $$
for $n\in \N$. These maps are quasisymmetries on $Q$ with $h_n\to h$ uniformly on $Q$. 

Note that we again have the relation \eqref{eq:inv}, which follows from the mapping properties of $\varphi$ and $T^{-k_n}$, in combination with the   identities $\xi(D_p)=D_p$ and $g_\xi^{-1}(D_p)=g_\xi(D_p)= D_p$ (the last relation follows from the complete invariance of $\mathcal{J}(g)$ under $g$).  As before,  \eqref{eq:inv} implies  \eqref{eq:inv2}.

Based on Theorem~\ref{thm:stabil} and Lemma~\ref{betaex} we can again argue that the sequence 
 $h_n$ stabilizes and so $h_{n+1}=h_n$ on $S_p$ for large $n$.  
 This implies that there exists $n\in \N$ such that 
 \begin{equation}\label{eq:hnstabcon} 
 g_\xi^{l_{n+1}}\circ T^{-k_{n+1}} =g_\xi^{l_n}\circ T^{-k_n}
 \end{equation} on $Y\cap S_p$. 
 Using the consistency relation for the inverse branches we see that 
 \begin{equation}\label{eq:gid} 
  g_\xi^{l_{n+1}} =g_\xi^{l_n}\circ T^{k_{n+1}-k_n}
  \end{equation} 
on $T^{-k_{n+1}}(Y\cap S_p)\sub D_p$. 

We want to argue that this identity remains valid on the whole set $D_p$. To see this, first note that  by Lemma~\ref{lem:qradm} for each $l\in \N$ the  iterate $g^l$ of $g$ is an  admissible maps  for the Sier\-pi\'n\-ski carpet $\mathcal{J}(g)$. Thus, $g_\xi^l =\xi^{-1}\circ g^l \circ  \xi$ is an admissible map
for $D_p$. Combined with Lemma~\ref{lem:compadm} and Lemma~\ref{lem:Tadm} this shows that both sides in \eqref{eq:gid} are admissible maps for $D_p$. 

Corollary~\ref{cor:id} then implies that \eqref{eq:gid} is valid on $D_p$. This is equivalent to a relation of the   form \eqref{eqn:rat} as required. This completes the proof when $D_p$ is quasisymmetrically equivalent to $\mathcal{J} (g)$.

Now suppose that there exists a quasisymmetry $\xi \: S_p\ra \mathcal J(g)$. Again we may assume that $\xi$ has an extension to an orientation-preserving quasisymmetry $\xi\: P \ra P$. 
Here we cannot expect   an identity as in  \eqref{eqn:rat} to be valid on $S_p$. The main problem is that $S_p$ is not 
forward-invariant under $T$. 

In order to derive  a contradiction, we have to slightly modify the above argument. We again implement a ``blow down-blow-up" 
procedure  as above, where $D_p$ is replaced with $S_p$, up to the point where we conclude that the sequence $\{h_n\}$ stabilizes.  We again obtain the relation   \eqref{eq:hnstabcon} on $Y\cap S_p$, where $Y\sub Q$ is a suitable good tile.  Instead of using the consistency relation 
\eqref{eq:consist} we now employ the identity 
$$ T^{-k_{n+1}} =T^{-(k_{n+1}-k_n)}\circ T^{-k_n}$$ 
for the  unique  branch $T^{-(k_{n+1}-k_n)}$ that maps $Q$ to the  good tile $Z\sub Q$ of level $k=k_{n+1}-k_n$ with $c\in Z$. This implies that there exist constants 
$l,l'\in \N$ such that 
  \begin{equation} \label{eq:SpJulia} 
  g_\xi^{l'}\circ T^{-k} =g_\xi^{l}
  \end{equation} 
on the  set $T^{-k_n}(Y\cap S_p)=Y'\cap S_p,$ where
 $Y'\coloneqq T^{-k_n}(Y)\sub Q$ is a good tile. 
 
Let $C$ and $C'$ be the finite sets of critical points of $g_\xi^{l}$ and 
$g_\xi^{l'}$, respectively.  If we define $W\coloneqq Q\setminus (O\cup C\cup T^k(C'))$, then  $W\cap S_p$ is connected and dense   in $S_p$. 
Moreover, each point in $x\in W\cap S_p$ is a 
good point (as defined in Section~\ref{sec:adm}) for each of the two maps   in \eqref{eq:SpJulia} and the Sier\-pi\'n\-ski carpet $S_p$.  So if we conjugate  these maps by $\beta$ from
Lemma~\ref{betaex}, then we obtain  Schottky maps 
from the locally porous relative Schottky set $\beta(W)\cap  \beta(S_p)$ into 
$\beta(S_p)$. By Theorem~\ref{thm:uniq}  this implies that \eqref{eq:SpJulia} holds on $W\cap S_p$. Since $W\cap S_p$ is dense in $S_p$, it follows that  \eqref{eq:SpJulia} is valid on $S_p$. 

We want to see that this is impossible. Since the union of all peripheral circles of $\mathcal{J}(g)$ is completely invariant under $g$, the union of all peripheral circles of $S_p$ is completely  invariant  under 
$g_\xi$.  
Now  consider 
 the
point $a\coloneqq (1, 1)\in S_p$. Then 
$T^{-k} (a)=(p^{-k},p^{-k})$ does not lie on a peripheral circle of $S_p$ and so the same is true for   
$b\coloneqq (g_\xi^{l'}\circ T^{-k}) (a)$.  On the other hand, $a$ lies on the peripheral circle $O$ of $S_p$ and so by  \eqref{eq:SpJulia}, 
 $$b= (g_\xi^{l'}\circ T^{-k}) (a)= g_\xi^{l}(a)\in g^l_\xi(O)$$  lies on the peripheral circle $g^l_\xi(O)$ of $S_p$.
 This is a contradiction. 
 
 We conclude that neither $D_p$ nor $S_p$ can be quasisymmetrically equivalent to the Julia set of a postcritically-finite rational map. 
 \end{proof} 

The essential point in the previous proof was   the fact that while the union of all peripheral circles of $\mathcal J(g)$ is completely invariant under $g$, the union of all peripheral circles of $S_p$ or $D_p$ is not completely invariant under $T$.

\end{document}